\documentclass[3p]{elsarticle}

\usepackage{math-nate}
\usepackage{graphicx}
\usepackage{ifpdf}
\ifpdf
\fi

\newcommand{\grad}{\nabla}
\usepackage{cancel}
\usepackage{hyperref}
\newcommand{\inner}[2]{\left\langle #1 , #2 \right\rangle}

\usepackage{url}
\numberwithin{equation}{section}
\usepackage{paralist}
\usepackage{enumitem}

\usepackage{verbatim} 
\usepackage{lineno}
\newcommand{\diffat}[2]{\left.\frac{d}{d #1}\right\vert_{#1 = #2}}
\newcommand{\asympon}[1]{\overset{#1}\asymp}

\newcommand{\funcclass}{\mathcal{C}}
\newcommand{\cylinder}{D}
\newcommand{\ignorethis}[1]{}
\newcommand{\haar}{m}
\newcommand{\groupop}{\star}

\newcommand{\heis}{\mathbb{H}_1}

\begin{document}

\begin{frontmatter}

\title{Gradient Estimates for the Subelliptic Heat
    Kernel on H-type Groups}
\author{Nathaniel Eldredge}
\address{Department of Mathematics, Cornell University, 593 Malott
  Hall, Ithaca, NY 14853 USA}
\ead{neldredge@math.cornell.edu}
\ead[url]{http://www.math.cornell.edu/~neldredge/}


\begin{abstract}
  We prove the following gradient inequality for the subelliptic heat
  kernel on nilpotent Lie groups $G$ of H-type:
  \begin{equation*}
    \abs{\grad P_t f} \le K P_t(\abs{\grad f})
  \end{equation*}
  where $P_t$ is the heat semigroup corresponding to the sublaplacian
  on $G$, $\grad$ is the subelliptic gradient, and $K$ is a constant.
  This extends a result of H.-Q. Li \cite{li-jfa} for the Heisenberg
  group.  The proof is based on pointwise heat kernel estimates, and
  follows an approach used by Bakry, Baudoin, Bonnefont, and Chafa\"i
  \cite{bbbc-jfa}.
\end{abstract}



\end{frontmatter}

\section{Introduction}

In \cite{li-jfa}, H.-Q. Li proved the following gradient inequality
for the heat kernel on the classical Heisenberg group $\heis$ of real
dimension $3$:
\begin{equation}\label{intro-grad-ineq}
  \abs{\grad P_t f} \le K P_t(\abs{\grad f})
\end{equation}
where $P_t$ is the heat semigroup corresponding to the usual
sublaplacian on $\heis$, $\grad$ is the corresponding
subgradient, $K$ is a constant, and $f$ is any appropriate smooth
function on $\heis$.  This was the first extension of
(\ref{intro-grad-ineq}) to a subelliptic setting; the elliptic case
was shown by Bakry \cite{bakry-riesz-notes-ii}, \cite{bakry-sobolev},
and in the case of a Riemannian manifold corresponds to a lower bound
on the Ricci curvature.



The proof in \cite{li-jfa} relies on pointwise upper and lower
estimates for the heat kernel, and a pointwise upper estimate for its
gradient, both of which were obtained in \cite{li-heatkernel} in the
context of Heisenberg groups of any dimension.  \cite{bbbc-jfa}
contains two alternate proofs of (\ref{intro-grad-ineq}) for the
classical Heisenberg group $\heis$, also depending on the pointwise heat
kernel estimates from \cite{li-heatkernel}.  Earlier, Driver and
Melcher in \cite{driver-melcher} had shown a partial result: that for
any $p > 1$ there exists a constant $K_p$ such that
\begin{equation}\label{intro-Lp-ineq}
  \abs{\grad P_t f}^p \le K_p P_t(\abs{\grad f}^p).
\end{equation}
Their argument proceeded probabilistically via methods of Malliavin
calculus and did not depend on heat kernel estimates, but they also
showed that it could not produce (\ref{intro-grad-ineq}),
which is the corresponding estimate with $p=1$.  \cite{tai-thesis}
extended the ``$L^p$-type'' inequality (\ref{intro-Lp-ineq}) to the case of a general
nilpotent Lie group, at the cost of replacing the constant $K_p$ with
a function $K_p(t)$.

In \cite{eldredge-precise-estimates}, we were able to show that
pointwise heat kernel estimates analogous to those of
\cite{li-heatkernel} (see (\ref{p1-estimate}--\ref{gradzp1-estimate}))
hold for Lie groups of H type, a class which generalizes the
Heisenberg groups while retaining some rather strong algebraic
properties.  (H-type groups were introduced by Kaplan in
\cite{kaplan80}; a useful reference and primer is Chapter 18 of
\cite{blu-book}.)  The purpose of the present article is to show that
given these heat kernel estimates, the first proof from \cite{bbbc-jfa}
can be adapted to establish the inequality (\ref{intro-grad-ineq}) in
the setting of H-type groups.  Our proof approximately follows the
structure of the first proof from \cite{bbbc-jfa} but may be read
independently of it, and is more explicitly detailed.

\section{Definitions and notation}


In order to fix notation, we give a definition of H-type groups and
accompanying concepts.  Our notation, where applicable, matches that
of \cite{eldredge-precise-estimates}.

A finite-dimensional Lie algebra $\mathfrak{g}$ (with nonzero center
$\mathfrak{z}$), together with an inner product $\inner{\cdot}{\cdot}$, is said to be of \emph{H type} or \emph{Heisenberg
  type} if the following conditions hold:
\begin{enumerate}
\item $[\mathfrak{z}^\perp, \mathfrak{z}^\perp] = \mathfrak{z}$; and
\item For each $z \in \mathfrak{z}$, the map $J_z :
  \mathfrak{z}^\perp \to \mathfrak{z}^\perp$ defined by
  \begin{equation}\label{Jz-relation}
    \inner{J_z x}{y} = \inner{z}{[x,y]} \quad \text{for $x,y \in \mathfrak{z}^\perp$}
  \end{equation}
  is an orthogonal map when $\inner{z}{z}=1$.
\end{enumerate}
A connected, simply connected Lie group $G$ is said to be of H type if
its Lie algebra $\mathfrak{g}$ is equipped with an inner product
satisfying the above conditions.

It is easy to see that an H-type Lie algebra (respectively, Lie group) is
a step 2 stratified nilpotent Lie algebra (Lie group).  The special case
$m=1$ produces the isotropic Heisenberg or Heisenberg-Weyl groups, and
the case $n=m=1$ gives the classical Heisenberg group $\heis$ of dimension $3$
discussed in \cite{bbbc-jfa}.

As usual, $G$ can be identified as a set with $\mathfrak{g}$, taking
the exponential map to be the identity.  By fixing an orthonormal
basis for $\mathfrak{g} = \mathfrak{z}^\perp \oplus \mathfrak{z}$, we
can identify $G$ and $\mathfrak{g}$ with Euclidean space equipped with
an appropriate bracket, as the following proposition states.  (The proof
is uncomplicated.)

\begin{proposition}
  If $G$ is an H-type Lie group identified with its Lie algebra
  $\mathfrak{g}$, then there exist integers $n, m > 0$, a bracket
  operation $[\cdot,\cdot]$ on $\R^{2n+m} = \R^{2n} \times \R^m$, and a map $T :
  G \to \R^{2n+m}$ such that $T : \mathfrak{g} \to (\R^{2n+m}, [\cdot,\cdot])$
  is a Lie algebra isomorphism, $T \mathfrak{z} = 0 \times \R^m$, and
  $T$ is an isometry with respect to the inner product
  $\inner{\cdot}{\cdot}$ on $\mathfrak{g}$ and the usual Euclidean
  inner product on $\R^{2n+m}$.  If we define a group operation
  $\groupop$ on $\R^{2n+m}$ as usual via $v \groupop w = v + w +
  \frac{1}{2}[v,w]$, then $T : G \to (\R^{2n+m}, \groupop)$ is a Lie
  group isomorphism, which maps the center of $G$ to $0 \times
  \R^{m}$.  The identity of $G$ is $0$ and the group inverse is given
  by $g^{-1}=-g$.
\end{proposition}

Henceforth we make this identification, and assume that our Lie group
$G$ is just $\R^{2n+m}$ with an appropriate bracket $[\cdot,\cdot]$
and corresponding group operation $\groupop$.  We let $\{e_1, \dots,
e_{2n}\}$ denote the standard orthonormal basis for $\R^{2n} \times 0
\subset G$, and $\{u_1, \dots, u_m\}$ the standard orthonormal basis
for $0 \times \R^m \subset G$, and write elements of $G$ as $g = (x,z)
= \sum_i x^i e_i + \sum_j z^j u_j$.  The maps $J_z$ can then be
identified with skew-symmetric $2n \times 2n$ matrices, which are
orthogonal when $\abs{z} = 1$.

We remark a few obvious consequences of (\ref{Jz-relation}):
\begin{proposition} \label{Jz-prop}
\begin{enumerate}
\item $J_z$ depends linearly on $z$;
\item $\abs{J_z x} = \abs{z} \abs{x}$, and by polarization
  $\inner{J_z x}{J_w x} = \inner{z}{w}\abs{x}^2$ and $\inner{J_z
  x}{J_z y} = \abs{z}^2 \inner{x}{y}$;
\item $\inner{J_z x}{x} = 0$, so $J_z^* = -J_z$. \label{Jz-skew}
\item $J_z^2 = -\abs{z}^2 I$. \label{Jz-square}
\end{enumerate}
\end{proposition}

We note that Lebesgue measure $\haar$ on $\R^{2n+m} = G$ is
bi-invariant under the group operation, and thus $\haar$ can be taken
as the Haar measure on the locally compact group $G$.

For $i = 1, \dots, 2n$, let $X_i$ be the unique left-invariant vector
field on $G$, and $\hat{X}_i$ the unique right-invariant vector field,
such that $X_i(0) = \hat{X}_i(0) = \frac{\partial}{\partial x^i}$.  We
can write
\begin{equation}\label{Xi-dds}
  X_i f(g) = \diffat{s}{0} f(g \groupop (s e_i, 0)),
  \quad   \hat{X}_i f(g) = \diffat{s}{0} f((s e_i,
  0) \groupop g).
\end{equation}
A
straightforward calculation shows
\begin{equation}\label{Xi-formula}
  \begin{split}
    X_i &= \frac{\partial}{\partial x^i} + \frac{1}{2} \sum_{j=1}^m
    \inner{J_{u_j}x}{e_i} \frac{\partial}{\partial z^j} \\
    \hat{X}_i &= \frac{\partial}{\partial x^i} - \frac{1}{2} \sum_{j=1}^m
    \inner{J_{u_j}x}{e_i} \frac{\partial}{\partial z^j}
  \end{split}
\end{equation}
We note that $[X_i, \hat{X}_j] = 0$ for all $i,j$.

As a consequence of the H-type property, the collection $\{X_i(g),
[X_i, X_j](g) : i,j = 1, \dots, 2n\} \subset T_g G$ spans $T_g G$ for
each $g \in G$.  Such a collection is said to be
\emph{bracket-generating}.

The left-invariant \emph{subgradient} $\grad$ on $G$ is given by
$\grad f = (X_1 f, \dots, X_{2n} f)$, with the right-invariant
$\hat{\grad}$ defined analogously.  We shall also use the notation
$\grad_x f := \left( \frac{\partial}{\partial x^1}f,\dots,
\frac{\partial}{\partial x^{2n}}f\right)$ and $\grad_z f := \left(
\frac{\partial}{\partial z^1}f,\dots, \frac{\partial}{\partial
  z^m}f\right)$ to denote the usual Euclidean gradients in the $x$ and
$z$ variables, respectively.  Note that $\grad_z$ is both left- and
right-invariant.  From (\ref{Xi-formula}) it is easy to verify that
\begin{equation}\label{gradient-combos}
  \begin{split}
    \grad f(x,z) &= \grad_x f(x,z) + \frac{1}{2} J_{\grad_z f(x,z)} x
    \\
    \hat{\grad} f(x,z) &= \grad_x f(x,z) - \frac{1}{2} J_{\grad_z f(x,z)} x
  \end{split}
\end{equation}
In particular, since $J_z$ depends linearly on $z$ and is orthogonal
for $\abs{z}=1$, we have 
\begin{equation}\label{abs-grad-difference}
\abs{(\grad -
  \hat{\grad})f(x,z)} = \abs{x} \abs{\grad_z f(x,z)}.
\end{equation}
We shall make use of this fact later.

The left-invariant \emph{sublaplacian} $L$ is the second-order
differential operator defined by $L = X_1^2 + \dots + X_{2n}^2$; $L$
is subelliptic but not elliptic.  By a renowned theorem due to
H\"ormander \cite{hormander67}, the bracket-generating condition
implies that $L$ is hypoelliptic, so that if $L f \in C^\infty$ then $f
\in C^\infty$; the same holds for the heat operator $L -
\frac{\partial}{\partial t}$.  $L$ is an essentially self-adjoint
operator on $L^2(\haar)$, and we let $P_t := e^{tL}$ be the heat semigroup
corresponding to $L$. $P_t$ has a convolution kernel $p_t$, so that
\begin{equation}\label{convolution}
P_t f(g) =
\int_G f(g \groupop k) p_t(k)\, d\haar(k).
\end{equation}
By hypoellipticity, $p_t$ is a smooth function on $G$.  An explicit
formula for $p_t$ is known:
\begin{equation}\label{pt-formula}
  p_t(x,z) = (2\pi)^{-m}(4\pi)^{-n} \int_{\R^m}
  e^{i\inner{\lambda}{z}-\frac{1}{4}\abs{\lambda} \coth(t
    \abs{\lambda}) \abs{x}^2} \left(\frac{\abs{\lambda}}{\sinh( t
    \abs{\lambda})}\right)^{n} \,d\lambda.
\end{equation}
See, among others, \cite{randall} for a derivation of
(\ref{pt-formula}).  We note in particular that $p_t$ is a radial
function; i.e. $p_t(x,z)$ is a function of $\abs{x}, \abs{z}$.  This
is unsurprising in light of the fact, easily verified, that $L$ maps
radial functions to radial functions.

For $\alpha > 0$, define the \emph{dilation} $\varphi_\alpha : G \to
G$ by $\varphi_\alpha(x,z) = (\alpha x, \alpha^2 z)$; then
$\varphi_\alpha$ is a group automorphism of $G$.  A straightforward
computation shows that $X_i (f \circ \varphi_\alpha) = \alpha (X_i f)
\circ \varphi_\alpha$, and $P_t(f \circ \varphi_\alpha) = (P_{\alpha^2
  t} f) \circ \varphi_\alpha$.

We now make some definitions concerning the geometry of $G$.  An
absolutely continuous path $\gamma : [0,1] \to G$ is said to be
\emph{horizontal} if there exist absolutely continuous $a_i : [0,1]
\to \R$ such that $\dot{\gamma}(t) = \sum_{i=1}^{2n} a_i(t)
X_i(\gamma(t))$.  In such a case the \emph{speed} of $\gamma$ is given
by $\norm{\dot{\gamma}(t)} := \left(\sum_{i=1}^{2n}
a_i(t)^2\right)^{1/2}$.  (This corresponds to taking a subriemannian
metric on $G$ such that $\{X_i\}$ are an orthonormal frame for the
horizontal bundle; see \cite{montgomery} for an exposition of these
ideas from subriemannian geometry.)  The \emph{length} of $\gamma$ is
defined as $\ell[\gamma] := \int_0^1 \norm{\dot{\gamma}(t)}\,dt$.  The
\emph{Carnot-Carath\'eodory distance} between two points $g, h \in G$
is
\begin{equation*}
  d(g,h) := \inf\left\{\ell[\gamma] : \gamma \text{ horizontal}, \gamma(0)
  = g, \gamma(1) = h\right\}.
\end{equation*}
By the left-invariance of the vector fields $X_i$, it follows that
$d(g,h) = d(kg, kh)$.  

By Chow's theorem, the bracket-generating condition implies that
$d(g,h) < \infty$ for all $g,h \in G$.  An explicit formula for $d$
and for length-minimizing paths (geodesic) can be found in
\cite{eldredge-precise-estimates}.  For the moment we note that $d(0,
(x,z)) \asymp \abs{x} + \abs{z}^{1/2}$, where the symbol $\asymp$ is
defined as follows.

\begin{notation}
  If $X$ is a set, and $a, b : X \to \R$ are real-valued functions on
  $X$, we write $a \asymp b$ to mean that there exist positive finite
  constants $C_1, C_2$ such that $C_1 b(x) \le a(x) \le C_2 b(x)$ for
  all $x \in X$.  We will also write $a \asympon{X} b$ if the domain
  where the estimates hold is not obvious from context.
\end{notation}

We will make extensive use of the following precise pointwise estimates on the
heat kernel $p_t$, which were obtained in
\cite{eldredge-precise-estimates} by using the explicit formula (\ref{pt-formula}):
\begin{align}
  p_1(x,z) &\asymp \frac{1 + (d(0,(x,z)))^{2n-m-1}}{1+(\abs{x}d(0,(x,z)))^{n-\frac{1}{2}}}
  e^{-\frac{1}{4}d(0,(x,z))^2} \label{p1-estimate}\\
  \abs{\grad p_1(x,z)} &\le C(1+d(0,(x,z)))p_1(x,z) \label{gradp1-estimate} \\
  \abs{\grad_z p_1(x,z)} &\le C p_1(x,z). \label{gradzp1-estimate}
  \intertext{We can combine (\ref{gradp1-estimate}) and
    (\ref{gradzp1-estimate}) using (\ref{abs-grad-difference}) to
    obtain}
   \abs{\hat{\grad} p_1(x,z)} &\le C(1+d(0,(x,z)))p_1(x,z). \label{gradhatp1-estimate} 
\end{align}

Let $\funcclass$ be the class of $f \in C^1(G)$ for which there exist
constants $M \ge 0$, $a \ge 0$, and $\epsilon \in (0,1)$ such that
\begin{equation*}
  \abs{f(g)} + \abs{\grad f(g)} + \abs{\hat{\grad} f(g)} \le M e^{a d(0,g)^{2-\epsilon}}
\end{equation*}
for all $g \in G$.  By the heat kernel bounds (\ref{p1-estimate}), the
convolution formula (\ref{convolution}) makes sense for all $f \in
\funcclass$, and thus we shall treat (\ref{convolution}) as the
definition of $P_t f$ for $f \in \funcclass$.  It is easy to see, by
the translation invariance of the Haar measure $\haar$, that $P_t$
remains left invariant under this definition.

The main theorem of this article is the following:
\begin{theorem}\label{main-grad-theorem}
  There exists a finite constant $K$ such that for all $f \in
  \funcclass$,
  \begin{equation}\label{grad-ineq}
    \abs{\grad P_t f} \le K P_t(\abs{\grad f}).
  \end{equation}
\end{theorem}

Following an argument found in \cite{driver-melcher}, by
left-invariance of $P_t$ and $\grad$, we see that in order to
establish (\ref{grad-ineq}) it suffices to show that it holds at the
identity, i.e. to show
\begin{equation}\label{grad-ineq-identity}
\abs{(\grad P_t f)(0)} \le K P_t(\abs{\grad f})(0).
\end{equation}
It also suffices to assume $t=1$.  This can be seen by taking $t=1$ in
 (\ref{grad-ineq-identity}) and replacing $f$ by $f \circ \varphi_{s^{1/2}}$.    


Therefore, in order to prove Theorem \ref{main-grad-theorem}, it will
suffice to show $\abs{(\grad P_1 f)(0)} \le K P_1(\abs{\grad f})(0)$.
We may replace $\grad$ by $\hat{\grad}$ on the left side, since $\grad
= \hat{\grad}$ at $0$.  Since $[X_i, \hat{X}_j] =
0$, we expect that $\hat{\grad}$ should commute with $P_t$, which we
now verify.

\begin{proposition}\label{commute}
  For $f \in \funcclass$, $\hat{\grad} P_t f(0)=(P_t \hat{\grad}f )(0)$.
\end{proposition}
\begin{proof}
  By (\ref{Xi-dds}) and (\ref{convolution}) we have
  \begin{align*}
    \hat{X}_i P_t f(0) &= \diffat{s}{0} P_t f(s
    e_i, 0) \\
    &= \diffat{s}{0} \int_G f((s e_i, 0)
    \groupop k) p_t(k)\,d\haar(k).
  \end{align*}
  We now differentiate under the integral sign, which can be justified
  because
  \begin{align*}
    \abs{\frac{d}{ds} f((s e_i, 0) \groupop k)} &=
    \abs{\diffat{\sigma}{0} f(((s+\sigma)e_i, 0)
    \groupop k)} \\
    &= \abs{\diffat{\sigma}{0} f((\sigma e_i, 0)
    \groupop (s e_i, 0) \groupop k)} \\
    &= \abs{\hat{X_i} f((s e_i, 0) \groupop k)} \\
    &\le M e^{a d(0, (s e_i, 0) \groupop k)^{2-\epsilon}}.
  \end{align*}
  But 
  \begin{align*}
    d(0, (s e_i, 0) \groupop k) &= d((s e_i, 0)^{-1}, k) = d((-s e_i,
    0), k) \\
    &\le d(0, (-s e_i, 0)) + d(0, k) = \abs{s} + d(0,k).
\end{align*}
  Thus for all
  $s \in [-1,1]$ we have
  \begin{equation*}
    \abs{\frac{d}{ds} f((s e_i, 0) \groupop k)} \le M e^{a (1 +
      d(0,k))^{2-\epsilon}} \le M' e^{a' d(0,k)^{2-\epsilon}}
  \end{equation*}
  for some $M', a'$, and therefore by the heat kernel bounds
  (\ref{p1-estimate}) we have
  \begin{equation*}
    \int_G \sup_{s \in [-1,1]} \abs{\frac{d}{ds} f((s e_i, 0) \groupop
      k)} p_t(k)\,d\haar(k) < \infty
  \end{equation*}
  which justifies differentiating under the integral sign.  Thus
  \begin{align*}
    \hat{X}_i P_t f(0) &= \int_G \diffat{s}{0}  f((s e_i, 0)
    \groupop k) p_t(k)\,d\haar(k) \\
    &= \int_G \hat{X}_i f(k)
    p_t(k)\,d\haar(k) \\
    &= P_t \hat{X}_i f(0).
  \end{align*}
  This completes the proof.
\end{proof}

Thus Theorem
\ref{main-grad-theorem} reduces to showing
\begin{equation}\label{grad-ineq-hat-semigroup}
  \abs{(P_1 \hat{\grad} f)(0)} \le K P_1(\abs{\grad f})(0)
\end{equation}
or in other words
\begin{equation}
  \abs{\int_G (\hat{\grad} f) p_1\,d\haar} \le K \int_G \abs{\grad f} p_1\,d\haar
\end{equation}
for which it suffices to show
\begin{equation}\label{integral-to-show}
 \abs{\int_G ((\grad-\hat{\grad}) f) p_1\,d\haar} \le K \int_G \abs{\grad f} p_1\,d\haar.
\end{equation}

A similar argument can be used to verify the following integration by
parts formula.
\begin{proposition}
  If $f \in \funcclass$, then
  \begin{equation}\label{byparts}
    \begin{split}
      \int_G (\grad f) p_1\,d\haar &= - \int_G (\grad p_1) f\,d\haar \\
      \int_G (\hat{\grad} f) p_1\,d\haar &= - \int_G (\hat{\grad} p_1) f\,d\haar
    \end{split}
  \end{equation}
\end{proposition}
\begin{proof}
  Tentatively, we have
  \begin{align*}
    \int_G ((X_i f) p_1 + f X_i p_1)\,d\haar &= \int_G X_i(f
    p_1)\,d\haar \\
    &= \int_G \diffat{s}{0} (f p_1)(g \groupop (s e_i, 0))\,d\haar(g)
    \\
    &\overset{?}{=} \diffat{s}{0} \int_G  (f p_1)(g \groupop (s e_i, 0))\,d\haar(g)
    \\
    &= \diffat{s}{0} \int_G  (f p_1)(g)\,d\haar(g) = 0    
  \end{align*}
  by right invariance of Haar measure $\haar$.  It remains to justify
  the differentiation under the integral sign in the third line.  We
  note that
  \begin{align*}
    \int_G \sup_{s \in [-1,1]} \abs{\frac{d}{ds} (f p_1)(g \groupop (s
      e_i, 0))}\,d\haar(g) &= \int_G \sup_{s \in [-1,1]} \abs{X_i (f p_1)(g \groupop (s
      e_i, 0))}\,d\haar(g) \\
    &\le \int_G \sup_{s \in [-1,1]} \abs{((X_i f) p_1)(g \groupop (s
      e_i, 0))}\,d\haar(g) \\ &\quad + \int_G \sup_{s \in [-1,1]}
    \abs{(f X_i p_1)(g \groupop (s
      e_i, 0))}\,d\haar(g).
  \end{align*}
  The first integral is easily seen to be finite by the definition of
  $\mathcal{C}$ and the heat kernel estimate (\ref{p1-estimate}), by
  similar logic to that in the proof of Proposition \ref{commute}.
  The second integral is similar; we may bound $\abs{\grad p_1}$ using
  the estimates (\ref{gradp1-estimate}) and (\ref{p1-estimate}).

  To show the second identity, involving $\hat{\grad}$, the same
  argument applies, using instead the left invariance of Haar
  measure.  We can bound $\abs{\hat{\grad} p_1}$ using
  (\ref{gradhatp1-estimate}) and (\ref{p1-estimate}).
\end{proof}

We now introduce an alternate coordinate system on $G$, similar but
not exactly analogous to the so-called ``polar coordinate'' system
used in \cite{bbbc-jfa}.  As shown in
\cite{eldredge-precise-estimates}, there is a unique (up to
reparametrization) shortest horizontal path or \emph{geodesic} from
the identity $0$ to each point $(x,z) \in G$ with $x,z$ nonzero; it
has as its projection onto $\R^{2n} \times 0$ an arc of a circle lying
in the plane spanned by $x$ and $J_z x$, with the origin as one
endpoint, and $x$ as the other.  The region in this plane bounded by
the arc and the straight line from $0$ to $x$ has area equal to
$\abs{z}$.  The projection of the geodesic onto $0 \times \R^m$ is a
straight line from $0$ to $z$.

Our new coordinate system will identify a point $(x,z)$ with the point
$u \in \R^{2n}$ which is the center of the arc, and a vector $\eta \in
\R^m$ which is parallel to $z$ and whose magnitude equals the angle
subtended by the arc.  The change of coordinates $(u, \eta) \mapsto
(x,z)$ will be denoted by
\begin{align}
  \Phi &: \{(u,\eta) \in \R^{2n+m} : 0 < \abs{\eta} < 2\pi\}  \to
  \{(x,z) \in G : x \ne 0, z \ne 0\} \label{Phi-domain}
\intertext{where}
  \Phi(u,\eta) &:= \left( \left(I - e^{J_{\eta}}\right) u,
  \frac{\abs{u}^2}{2}
  \left(1-\frac{\sin\abs{\eta}}{\abs{\eta}}\right)\eta\right) \\
  &= \left( (1-\cos\abs{\eta})u + \frac{\sin\abs{\eta}}{\abs{\eta}}
  J_\eta u,     \frac{\abs{u}^2}{2}
  \left(1-\frac{\sin\abs{\eta}}{\abs{\eta}}\right)\eta\right)
\end{align}
by Proposition \ref{Jz-prop}, items \ref{Jz-skew} and
\ref{Jz-square}.  

To visualize this, let us consider the special case of the
Heisenberg group $\heis$, with $n=m=1$.  It is convenient to
identify the subspace $\R^{2n} \times 0$ with $\C$; in this case,
$J_\eta = i \eta$.  (Note $\eta \in \R$.)  We then have
\begin{equation}
  \Phi(u, \eta) = \left( (1-e^{i\eta})u, \frac{\abs{u}^2}{2}(\eta -
  \sin \eta)\right). 
\end{equation}
See Figure \ref{phi-fig} for an illustration of the relationship
between $(u, \eta)$ and $\Phi(u,\eta)$ in $\heis$.

  \begin{figure}
    \centering
    \includegraphics{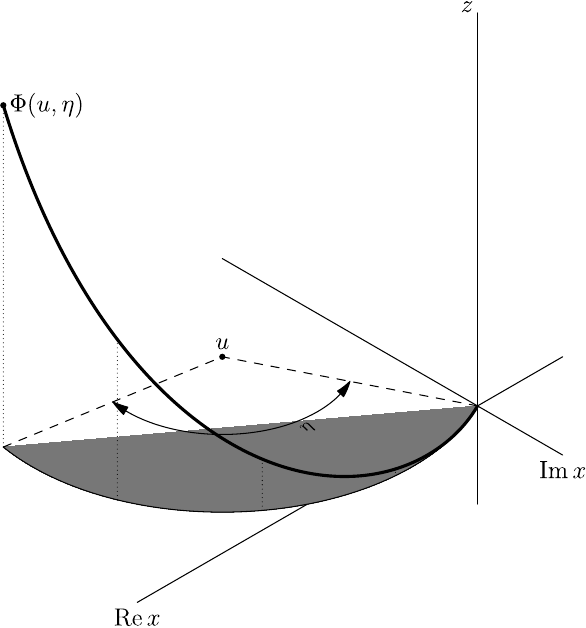}
    \caption{Illustration of the change of coordinates $\Phi$ in the
      classical Heisenberg group $\heis$.  The bold line is a
      geodesic, whose projection into the $x$-plane is
      an arc of a circle with center $u$ and subtending an angle
      $\eta$.  The $z$ coordinate of $\Phi(u,\eta)$ is equal to
      the area of the shaded circular segment.
    }\label{phi-fig}
  \end{figure}

Note that the $(u, \eta)$ coordinate system omits the set $\{z = 0\}
= \R^{2n} \times 0 \subset G$, for which the arc degenerates into a
straight line and has ``infinite radius,'' as well as the set $\{x =
0\} = 0 \times \R^{m}$, for which the arc becomes a circle whose
center $u$ is no longer uniquely determined.  These sets are of Haar
measure zero and hence will be neglected in the argument without
further comment.  Estimates which are shown to hold off these sets
will also hold on them, by continuity.

$\Phi$ has the property that for each $(u, \eta)$, the path $s \mapsto
\Phi(u,s\eta)$ traces the shortest horizontal path between any two of
its points, and has constant speed $\abs{u}\abs{\eta}$.  In
particular,
  \begin{equation}\label{dPhi}
    d(0, \Phi(u,\eta)) = \abs{u}\abs{\eta}.
    \end{equation}
Also, for any $f \in C^1(G)$,
\begin{equation}\label{geodesic-grad}
  \abs{\frac{d}{ds} f(\Phi(u, s\eta))} \le \abs{u} \abs{\eta}
  \abs{\grad f(\Phi(u, s\eta))}.
\end{equation}

  Note that if $(x,z) = \Phi(u,\eta)$, we have 
  \begin{align*}
    \abs{x}^2 &= \abs{u}^2 (2-2\cos\abs{\eta}) \\ 
    \abs{z} &= \frac{\abs{u}^2}{2} (\abs{\eta}-\sin\abs{\eta}).
\end{align*} 

To compare this with the ``polar coordinates'' $(u,s)$ used in
\cite{bbbc-jfa}, take $u=u$ and $s = \abs{u} \eta$.

Let $B := \{ g \in G : d(0, g) < 1\}$ denote the unit ball of the
Carnot-Carath\'eodory distance.  To express $B$ in $(u,\eta)$
coordinates, we note by (\ref{dPhi}) that $\Phi(u,\eta) \in B$ iff
$\abs{u}\abs{\eta} \le 1$; combining this with the constraints on $u$
and $\eta$ given in (\ref{Phi-domain}), we have
\begin{align}
  B &= \left\{ \Phi(u,\eta) : u \in \R^{2n}, \abs{\eta} < 2\pi \wedge
  \frac{1}{\abs{u}}\right\} \label{BPhi}
\intertext{and conversely}
B^C &= \left\{ \Phi(u, \eta) : \abs{u} \ge \frac{1}{2\pi},
\frac{1}{\abs{u}} \le \abs{\eta} < 2\pi \right\} \label{BCPhi}
\end{align}
modulo the null sets $\{x=0\}$ and $\{z=0\}$, as usual.

In $(u, \eta)$ coordinates, the heat kernel estimate
(\ref{p1-estimate}) reads
  \begin{align}
    p_1(\Phi(u,\eta)) &\asymp
    \frac{1+(\abs{u}\abs{\eta})^{2n-m-1}}{1+(\abs{u}^2 \abs{\eta}
      \sqrt{2-2\cos\abs{\eta}})^{n-\frac{1}{2}}}
    e^{-\frac{1}{4}(\abs{u}\abs{\eta})^2} \label{p1-u-eta-estimate-cos} \\
    &\asymp
    \frac{1+(\abs{u}\abs{\eta})^{2n-m-1}}{1+\left(\abs{u}^2 \abs{\eta}^2
      (2\pi-\abs{\eta}) \right)^{n-\frac{1}{2}}}
    e^{-\frac{1}{4}(\abs{u}\abs{\eta})^2} \label{p1-u-eta-estimate}
  \end{align}
  since $1-\cos \theta \asymp \theta^2(2\pi-\theta)^2$ for $\theta \in [0,2\pi]$.
We will often abuse notation and write $p_1(u,\eta)$ for
$p_1(\Phi(u,\eta))$, when no confusion will result.

\section{Proof of the gradient estimate}\label{sec-gradient-proof}

We now begin the proof of Theorem \ref{main-grad-theorem}, which
occupies the rest of this article.

We begin by computing the Jacobian determinant of the change of
coordinates $\Phi$, so that we can use $(u,\eta)$ coordinates in
explicit computations.

\begin{lemma}
  Let $A(u,\eta)$ denote the Jacobian determinant of $\Phi$, so
  that $d\haar = A(u, \eta)\,du\,d\eta$.  Then
  \begin{equation}
    A(u,\eta) = \abs{u}^{2m} \left(\frac{1}{2}-\frac{\sin\abs{\eta}}{2\abs{\eta}}\right)^{m-1}
      (2-2\cos\abs{\eta})^{n-1} \left(2 - 2\cos\abs{\eta} - \abs{\eta}\sin\abs{\eta}\right).
  \end{equation}
\end{lemma}
Note that $A(u,\eta)$ depends on $u, \eta$ only through their
absolute values $\abs{u}, \abs{\eta}$.  By an abuse of notation we may
occasionally use $A$ with $u$ or $\eta$ replaced by scalars, so that
$A(r, \rho)$ means $A(r \hat{u}, \rho \hat{\eta})$ for arbitrary unit
vectors $\hat{u}, \hat{\eta}$. 

For the Heisenberg group $\heis$ with $n=m=1$, this reduces to
\begin{equation*}
      A(u,\eta) = \abs{u}^{2} \left(2 - 2\cos\abs{\eta} - \abs{\eta}\sin\abs{\eta}\right).
\end{equation*}
The analogous expression appearing in \cite{bbbc-jfa} is slightly
incorrect.  However, it does have the same asymptotics as the correct
expression (see Corollary \ref{A-estimates}), which is sufficient for
the rest of the argument in \cite{bbbc-jfa}, so that its overall
correctness is not affected.

\begin{proof}
  Fix $u, \eta$.  Form an orthonormal basis for
  $T_{(u,\eta)}\Phi^{-1}(G) \cong \R^{2n+m}$ as follows.  Let
  $\hat{u}$ be a unit vector in the direction of $(u,0)$, $\hat{v}$ a
  unit vector in the direction of $(J_\eta u,0)$.  For $i=1, \dots,
  n-1$ let $\hat{w}_i, \hat{y}_i \in \R^{2n} \times 0$ be unit vectors
  such that $\hat{w}_i$ is orthogonal to $\hat{u}, \hat{v}, \hat{w}_j,
  \hat{y}_j, 1 \le j < i$, and let $\hat{y}_i$ be in the direction of
  $J_\eta \hat{w}_i$ so that $\hat{y}_i$ is orthogonal to $\hat{u},
  \hat{v}, \hat{w}_j, \hat{y}_j, 1 \le j < i$ as well as to
  $\hat{w}_i$.  (To see this, note that if $\inner{x}{y}=0$ and
  $\inner{x}{J_z y}=0$, then $\inner{J_z x}{y} = 0$ and $\inner{J_z
    x}{J_z y} = -\abs{z}^2\inner{x}{y}=0$.)  Let
  $\hat{\eta}$ be a unit vector in the direction of $(0,\eta)$, and
  let $\hat{\zeta}_k, k=1,\dots,m-1$ be orthonormal vectors in $0
  \times \R^{m}$ which are orthogonal to $\hat{\eta}$.  Then
  $\{\hat{u}, \hat{v}, \hat{w}_i, \hat{y}_i, \hat{\eta},
  \hat{\zeta}_k\}$ form an orthonormal basis for $\R^{2n+m}$.  Note
  $J_\eta \hat{u} = \abs{\eta} \hat{v}$, $J_\eta \hat{v} = -\abs{\eta}
  \hat{u}$,$J_\eta \hat{w}_i = \abs{\eta} \hat{y}_i$, $J_\eta
  \hat{y}_i = -\abs{\eta} \hat{w}_i$.  Then
  \begin{align*}
    \partial_{\hat{u}} \Phi(u,\eta) &= (1-\cos\abs{\eta})\hat{u}
    + \sin\abs{\eta} \hat{v} + \abs{u}
    \left(\abs{\eta}-\sin\abs{\eta}\right)\hat{\eta}\\
    \partial_{\hat{v}} \Phi(u,\eta) &= (1-\cos\abs{\eta})\hat{v}
    - \sin\abs{\eta} \hat{u} \\
    \partial_{\hat{w}_i} \Phi(u,\eta) &= (1-\cos\abs{\eta})\hat{w}_i
    + \sin\abs{\eta} \hat{y}_i \\
    \partial_{\hat{y}_i} \Phi(u,\eta) &= (1-\cos\abs{\eta})\hat{y}_i
    - \sin\abs{\eta} \hat{w}_i \\
    \partial_{\hat{\eta}} \Phi(u,\eta) &= \abs{u}(\sin\abs{\eta}) \hat{u} +
    \abs{u}(\cos\abs{\eta}) \hat{v} + \frac{\abs{u}^2}{2}
    \left(1-\cos\abs{\eta}\right)\hat{\eta} \\
    \partial_{\hat{\zeta}_k} \Phi(u,\eta) &=
    \frac{\sin{\abs{\eta}}}{\abs{\eta}} J_{\hat{\zeta}_k} u + \frac{\abs{u}^2}{2}
    \left(1-\frac{\sin\abs{\eta}}{\abs{\eta}}\right)\hat{\zeta}_k.
  \end{align*}
  In this basis, the Jacobian matrix has the form
  \begin{align}
    J &= 
    \begin{pmatrix}
      1 - \cos\abs{\eta} & -\sin\abs{\eta} & 0 & \abs{u}\sin\abs{\eta}
      & 0 \\
      \sin\abs{\eta} & 1-\cos\abs{\eta} & 0 & \abs{u}\cos\abs{\eta} &
      0 \\
      0 & 0 & B & 0 & * \\
      \abs{u}(\abs{\eta}-\sin\abs{\eta}) & 0 & 0 &
      \frac{\abs{u}^2}{2}(1-\cos\abs{\eta}) & 0 \\
      0 & 0 & 0 & 0 & D
    \end{pmatrix}_{(2n+m) \times (2n+m)}
    \intertext{where}
    B &:= \begin{pmatrix}
            1 - \cos\abs{\eta} & -\sin\abs{\eta} &  & & \\
            \sin\abs{\eta} & 1-\cos\abs{\eta} & & & \\
            & & \ddots & & \\
            & & & 1 - \cos\abs{\eta} & -\sin\abs{\eta} \\
            & & & \sin\abs{\eta} & 1-\cos\abs{\eta}
    \end{pmatrix}_{2(n-1) \times 2(n-1)}
    \intertext{ is a block-diagonal matrix of $2\times 2$ blocks, and}
    D &:=     \begin{pmatrix}
      \frac{\abs{u}^2}{2}
      \left(1-\frac{\sin\abs{\eta}}{\abs{\eta}}\right) & & \\
      & \ddots & \\
      & & \frac{\abs{u}^2}{2}
      \left(1-\frac{\sin\abs{\eta}}{\abs{\eta}}\right)
    \end{pmatrix}_{(m-1)\times (m-1)}
  \end{align}
  is diagonal.  Note $\abs{B} = (2-2\cos\abs{\eta})^{n-1}$ and $\abs{D} = \left(\frac{\abs{u}^2}{2}
      \left(1-\frac{\sin\abs{\eta}}{\abs{\eta}}\right)\right)^{m-1}$.

  So factoring out $\abs{D}$ and expanding about the $\hat{\eta}$ row,
  we have
  \begin{align*}
    \abs{J} &= \abs{D} \left(\abs{u}(\abs{\eta}-\sin\abs{\eta}) 
    \begin{vmatrix}
      -\sin\abs{\eta} & 0 & \abs{u}\sin\abs{\eta} \\
      1-\cos\abs{\eta} & 0 & \abs{u}\cos\abs{\eta} \\
      0 & B & 0       
    \end{vmatrix}
    + \frac{\abs{u}^2}{2}(1-\cos\abs{\eta}) 
    \begin{vmatrix}
      1 - \cos\abs{\eta} & -\sin\abs{\eta} & 0 \\
      \sin\abs{\eta} & 1-\cos\abs{\eta} & 0  \\
      0 & 0 & B 
    \end{vmatrix}
    \right) \\
    &= \left(\frac{\abs{u}^2}{2}
      \left(1-\frac{\sin\abs{\eta}}{\abs{\eta}}\right)\right)^{m-1} 
      \\ &\quad \times
      \left(\abs{u}(\abs{\eta}-\sin\abs{\eta})
      (-\abs{u}\sin\abs{\eta}) (2-2\cos\abs{\eta})^{n-1} + 
      \frac{\abs{u}^2}{2}(1-\cos\abs{\eta}) (2-2\cos\abs{\eta})^{n}
    \right)  \\
    &= \left(\frac{\abs{u}^2}{2}
      \left(1-\frac{\sin\abs{\eta}}{\abs{\eta}}\right)\right)^{m-1}
      \abs{u}^2 (2-2\cos\abs{\eta})^{n-1} \left((\abs{\eta}-\sin\abs{\eta})
      (-\sin\abs{\eta})  + 
      (1-\cos\abs{\eta})^2\right) \\
    &= \abs{u}^{2m} \left(\frac{1}{2}-\frac{\sin\abs{\eta}}{2\abs{\eta}}\right)^{m-1}
      (2-2\cos\abs{\eta})^{n-1} \left(2 - 2\cos\abs{\eta} - \abs{\eta}\sin\abs{\eta}\right)
  \end{align*}
\end{proof}


\begin{corollary}\label{A-estimates}
  \begin{equation}\label{A-estimates-eqn}
    A(u, \eta) \asymp \abs{u}^{2m} \abs{\eta}^{2(m+n)} (2\pi-\abs{\eta})^{2n-1}
  \end{equation}
\end{corollary}
\begin{proof}
  The asymptotic equivalence near $\abs{\eta}=0$ and $\abs{\eta}=2\pi$
  follows from a routine Taylor series computation.

  It then suffices to show that $A(u,\eta) > 0$ for all $0 <
  \abs{\eta} < 2\pi$.  We have $\frac{1}{2} -
  \frac{\sin\abs{\eta}}{2\abs{\eta}} > 0$ for all $\abs{\eta} > 0$,
  since $x > \sin x$ for all $x > 0$.  We also have $2-2\cos\abs{\eta}
  > 0$ for all $0 < \abs{\eta} < 2\pi$.

  Finally, to show $f(\abs{\eta}) := 2 - 2\cos\abs{\eta} -
  \abs{\eta}\sin\abs{\eta} > 0$, let $\theta = \frac{1}{2}\abs{\eta}$.
  Using double-angle identities, we have $f(2\theta) = 4 \sin\theta
  (\sin\theta - \theta\cos\theta)$.  For $0 < \theta < \pi$ we have
  $\sin\theta > 0$ so it suffices to show $g(\theta) :=
  \sin\theta-\theta\cos\theta > 0$.  But we have $g(0)=0$ and
  $g'(\theta) = \theta\sin\theta > 0$ for $0 < \theta < \pi$.
\end{proof}

The heat kernel estimates will be used to prove a technical lemma
regarding integrating the heat kernel along a geodesic.  The proof
requires the following simple fact from calculus, of which a close
relative appears in \cite{eldredge-precise-estimates}.

\begin{lemma}\label{gaussian-calculus}
  For any $q \in \R$, $a_0 > 0$ there exists a constant $C = C_{q,a_0}$ such that for any
  $a \ge a_0$ we have
  \begin{equation}
    \int_{t=1}^{t=\infty} t^q e^{-(at)^2}\,dt \le C \frac{1}{a^2} e^{-a^2}.
  \end{equation}
\end{lemma}

\begin{proof}
  Make the change of variables $s=t^2$ to get
  \begin{equation*}
    \int_{t=1}^{t=\infty} t^q e^{-(at)^2}\,dt = \frac{1}{2}
    \int_{s=1}^{s=\infty} s^{q'} e^{-a^2 s}\,ds
  \end{equation*}
  where $q' = \frac{q-1}{2}$.  For $q' \le 0$ (i.e. $q \le 1$), we
  have $s^{q'} \le 1$ and thus
\begin{equation*}
    \int_{s=1}^{s=\infty} s^{q'} e^{-a^2 s}\,ds \le  \int_{s=1}^{s=\infty} e^{-a^2 s}\,ds = \frac{1}{a^2}e^{-a^2}.
\end{equation*}
For $q' > 0$, notice that integration by parts gives
\begin{align*}
  \int_{s=1}^{s=\infty} s^{q'} e^{-a^2 s}\,ds &=
  \frac{1}{a^2}e^{-a^2} + \frac{q'}{a^2}
  \int_{s=1}^{s=\infty} s^{q'-1} e^{-a^2 s}\,ds \\ 
  &\le \frac{1}{a^2}e^{-a^2} + \frac{q'}{a_0^2}\int_{s=1}^{s=\infty} s^{q'-1} e^{-a^2 s}\,ds
\end{align*}
whereupon the result follows by induction.
\end{proof}



\begin{lemma}\label{geodesic-integral-bound}
  For each $q \in \R$ there exists a constant $C_q$ such that for
  all $u, \eta$ with $\Phi(u, \eta) \in B^C$, i.e. $\abs{u}\abs{\eta}
  \ge 1$,  we have
  \begin{align}
    \int_{t=1}^{t=\frac{2\pi}{\abs{\eta} }} p_1(u,t\eta) A(u,t\eta)
    t^{q}\,dt &\le  \frac{C_q}{(\abs{u} \abs{\eta})^2} p_1(u,\eta)
    A(u,\eta) \label{geodesic-integral-bound-stronger} \\
    &\le C_q  p_1(u,\eta) A(u,\eta). \label{geodesic-integral-bound-weaker}
  \end{align}
\end{lemma}

Note that (\ref{geodesic-integral-bound-weaker}) follows immediately
from the stronger statement (\ref{geodesic-integral-bound-stronger}),
since by assumption $\abs{u}\abs{\eta} \ge 1$.  In fact, we shall only
use (\ref{geodesic-integral-bound-weaker}) in the sequel.

\begin{proof}
  Assume throughout that $\abs{u}\abs{\eta} \ge 1$ and $0 < \abs{\eta}
  < 2\pi$.  (See (\ref{BCPhi}).)

  The proof involves the fact that a geodesic passes through (up to) three
  regions of $G$ in which the estimates for $p_1$ and $A$ simplify in
  different ways.  We define these regions, which partition
  $B^C$, as follows.  See Figure \ref{regions-fig}.
  \begin{figure}
    \centering
    \includegraphics{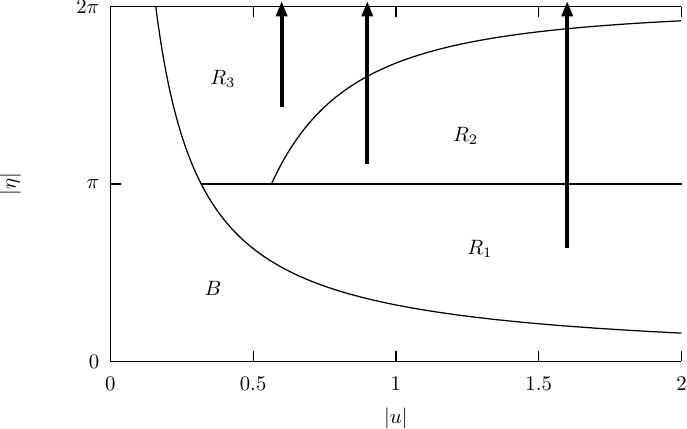}
    \caption{The regions $R_1, R_2, R_3$, seen in the
      $\abs{u}$-$\abs{\eta}$ plane.  The dark lines indicate examples of
      the geodesic paths of integration used in
      (\ref{geodesic-integral-bound-stronger}).  }\label{regions-fig}
  \end{figure}
  \begin{enumerate}
    \item Region $R_1$ is the set of $\Phi(u,\eta)$ such that
       $0 <
      \abs{\eta} \le \pi$.  (This corresponds to having $\abs{x}^2
      \lesssim \abs{z}$.) In this region we have $\abs{u} \ge
      \frac{1}{\pi}$ and $\pi \le 2\pi - \abs{\eta} < 2\pi$.
      Therefore (\ref{p1-u-eta-estimate}) becomes
      \begin{align*}
        p_1(u, \eta) &\asympon{R_1} (\abs{u}\abs{\eta})^{-m}
        e^{-\frac{1}{4}(\abs{u}\abs{\eta})^2}
        \intertext{and Corollary \ref{A-estimates} yields}
        A(u, \eta) &\asympon{R_1} \abs{u}^{2m} \abs{\eta}^{2(n+m)}
        \intertext{so that}
        p_1(u, \eta) A(u, \eta) &\asympon{R_1} \abs{u}^m
        \abs{\eta}^{2n+m} e^{-\frac{1}{4}(\abs{u}\abs{\eta})^2} =: F_1(u,\eta).
      \end{align*}

    \item Region $R_2$ is the set of $\Phi(u,\eta)$ such that
       $\pi <
      \abs{\eta} \le 2\pi-\frac{1}{\abs{u}^2}$.  (This corresponds to having
      $\abs{x}^2 \gtrsim \abs{z}$ and $\abs{x}^2 \abs{z} \gtrsim 1$.)
      In this region, we have $\abs{u}^2 \abs{\eta}^2
      (2\pi-\abs{\eta}) \ge \pi^2$, so that
      \begin{align*}
        p_1(u, \eta) &\asympon{R_2} \abs{u}^{-m}
        (2\pi-\abs{\eta})^{-n+\frac{1}{2}}
        e^{-\frac{1}{4}(\abs{u}\abs{\eta})^2} \\
        A(u, \eta) &\asympon{R_2} \abs{u}^{2m} (2\pi-\abs{\eta})^{2n-1}
        \\
        p_1(u, \eta) A(u, \eta) &\asympon{R_2} \abs{u}^m
        (2\pi-\abs{\eta})^{n-\frac{1}{2}}
        e^{-\frac{1}{4}(\abs{u}\abs{\eta})^2} =: F_2(u,\eta)
        \\ &\asympon{R_2}  \abs{u}^m \abs{\eta}^{2n+m}
        (2\pi-\abs{\eta})^{n-\frac{1}{2}}
        e^{-\frac{1}{4}(\abs{u}\abs{\eta})^2} =: \tilde{F}_2(u, \eta).
      \end{align*}
      We shall use the estimates $F_2, \tilde{F}_2$ at different
      times.  Although $F_2 \asympon{R_2} \tilde{F}_2$ (since
      $\abs{\eta} \asympon{R_2} 1$), they are not
      equivalent on $R_1$.

    \item Region $R_3$ is the set of $\Phi(u,\eta)$ such that
       $\abs{\eta} > \max\left(\pi,
      2\pi-\frac{1}{\abs{u}^2}\right)$.  (This corresponds to having
      $\abs{x}^2 \gtrsim \abs{z}$ and $\abs{x}^2 \abs{z} \lesssim 1$.)
      In this region, we have $\abs{u}^2 \abs{\eta}^2
      (2\pi-\abs{\eta}) < (2\pi)^2$, so that
      \begin{align*}
        p_1(u, \eta) &\asympon{R_3} \abs{u}^{2n-m-1}
        e^{-\frac{1}{4}(\abs{u}\abs{\eta})^2} \\
        A(u, \eta) &\asympon{R_3} \abs{u}^{2m} (2\pi-\abs{\eta})^{2n-1}
        \\
        p_1(u, \eta) A(u, \eta) &\asympon{R_3} \abs{u}^{2n+m-1}
        (2\pi-\abs{\eta})^{2n-1} e^{-\frac{1}{4}(\abs{u}\abs{\eta})^2}
        =: F_3(u,\eta)
      \end{align*}
  \end{enumerate}

  We observe that a geodesic starting from the origin (given by $t
  \mapsto \Phi(u, t\eta)$ for some fixed $(u,\eta)$) passes through
  these regions in order, except that it skips Region 2 if $\abs{u} <
  \pi^{-1/2}$.

  We now estimate the desired integral along a portion of a geodesic
  lying in a single region.
  \begin{claim}\label{single-region}
    Let $q \in \R$.  Suppose that $F : G \to \R$ is given by
    \begin{equation*}
      F(u,\eta) =  \abs{u}^{\alpha}
      \abs{\eta}^{\beta} (2\pi-\abs{\eta})^{\gamma}
      e^{-\frac{1}{4}(\abs{u}\abs{\eta})^2}
    \end{equation*}
    for some nonnegative powers $\alpha, \beta, \gamma$, and that
    there is some region $R \subset G$ such that $F \asympon{R} p_1
    A$.  Then there is a constant $C$ depending on
    $q$, $F$, $R$ such that for all $u, \eta, \tau_0, \tau_1, \tau_2$
    satisfying
    \begin{itemize}
    \item $\abs{u} \abs{\eta} \ge 1$;
    \item $1 \le \tau_0 \le \tau_1 \le \tau_2 \le
    \frac{2\pi}{\abs{\eta}}$; and
    \item $\Phi(u, t\eta) \in R$ for all $t \in
          [\tau_1,\tau_2]$
    \end{itemize}
    we have
    \begin{equation}
      \int_{t=\tau_1}^{t=\tau_2} p_1(u,t\eta) A(u, t\eta) t^q\,dt \le
      C \frac{\tau_0^{q-1}}{(\abs{u}\abs{\eta})^2} F(u,\tau_0 \eta).
    \end{equation}
  \end{claim}
  
\begin{proof}[Proof of Claim \ref{single-region}] We have
  \begin{align*}
    \int_{t=\tau_1}^{t=\tau_2} p_1(u,t\eta) A(u, t\eta) t^q\,dt &\le C
    \int_{t=\tau_1}^{t=\tau_2} F(u, t\eta) t^q\,dt \\
    &\le C \int_{t=\tau_0}^{t=\tau_2} F(u, t\eta) t^q\,dt \\
    &= C \abs{u}^{\alpha} \abs{\eta}^{\beta} \int_{t=\tau_0}^{t=\tau_2}
    t^{q+\beta} (2\pi-t\abs{\eta})^{\gamma}
    e^{-\frac{1}{4}(t\abs{u}\abs{\eta})^2}\,dt \\
    &\le C \abs{u}^{\alpha} \abs{\eta}^{\beta} (2\pi-\tau_0 \abs{\eta})^{\gamma}\int_{t=\tau_0}^{t=\tau_2}
    t^{q+\beta} 
    e^{-\frac{1}{4}(t\abs{u}\abs{\eta})^2}\,dt
\intertext{since $t \ge \tau_0$.  We now make the change of
  variables $t = t' \tau_0$:}
    &\le C \abs{u}^{\alpha} \abs{\eta}^{\beta}
(2\pi-\tau_0 \abs{\eta})^{\gamma} \tau_0^{q+\beta+1} \int_{t'=1}^{t'=\infty}
    t'^{q+\beta} 
    e^{-\frac{1}{4}(t' \tau_0 \abs{u}\abs{\eta})^2}\,dt' \\
    &\le C' \abs{u}^{\alpha} \abs{\eta}^{\beta}
    (2\pi-\tau_0 \abs{\eta})^{\gamma} \tau_0^{q+\beta+1} \frac{1}{(\tau_0
      \abs{u}\abs{\eta})^2} e^{-\frac{1}{4}(\tau_0 \abs{u}\abs{\eta})^2} \\
    &= C' \frac{\tau_0^{q-1}}{(\abs{u}\abs{\eta})^2} F(u, \tau_0 \eta)
  \end{align*}
  where in the second-to-last line we applied Lemma
  \ref{gaussian-calculus} with $a = \frac{1}{2} \tau_0
  \abs{u}\abs{\eta}$, \mbox{$a_0 = \frac{1}{2}$}.  (Note that
  $\abs{u}\abs{\eta} \ge 1$ and $\tau_0 \ge 1$ by assumption, so $a
  \ge a_0$.)  Claim \ref{single-region} is proved.
\end{proof}

  Now for fixed $u, \eta$, let
  \begin{align*}
    t_2 &:= \max\left(1, \frac{\pi}{\abs{\eta}}\right) \\
    t_3 &:= \max\left(t_2, \frac{1}{\abs{\eta}}\left(2\pi-\frac{1}{\abs{u}^2}\right)\right)
  \end{align*}
  so that
  \begin{align*}
    \Phi(u, t\eta) &\in R_1 \text{ for } 1 < t \le t_2 \\
    \Phi(u, t\eta) &\in R_2 \text{ for } t_2 < t < t_3 \\
    \Phi(u, t\eta) &\in R_3 \text{ for } t_3 \le t <
    \frac{2\pi}{\abs{\eta}}.
  \end{align*}

  We divide the remainder of the proof into cases, depending on the
  region where $\Phi(u,\eta)$ resides.

  \begin{case}
  Suppose that $\Phi(u,\eta) \in R_1$.  We have
  \begin{equation*}
    \int_{t=1}^{t=\frac{2\pi}{\abs{\eta}}} p_1(u,t\eta) A(u,t\eta)
    t^{q}\,dt = \int_{t=1}^{t=t_2} + \int_{t=t_2}^{t=t_3} +
    \int_{t=t_3}^{t=\frac{2\pi}{\abs{\eta}}}.
  \end{equation*}

  For the first integral, where $\Phi(u, t\eta) \in R_1$, we have by
  Claim \ref{single-region} (taking $\tau_0 = \tau_1 = 1$, $\tau_2 =
  t_2$, $R = R_1$, $F = F_1$) that
  \begin{align*}
    \int_{t=1}^{t=t_2} p_1(u,t\eta) A(u,t\eta) t^{q}\,dt \le 
    \frac{C}{(\abs{u}\abs{\eta})^2} F_1(u,\eta) \le \frac{C'}{(\abs{u}\abs{\eta})^2} p_1(u,\eta) A(u,\eta)
  \end{align*}
  since $F_1 \asympon{R_1} p_1 A$.  

  For the second integral, where $\Phi(u, t\eta) \in R_2$, we take
  $\tau_0 = 1$, $\tau_1 = t_2$, $\tau_2 = t_3$, $R = R_2$, $F =
  \tilde{F}_2$ in Claim \ref{single-region} to obtain
  \begin{align*}
    \int_{t=t_2}^{t=t_3} p_1(u,t\eta) A(u,t\eta) t^{q}\,dt 
    &\le  \frac{C}{(\abs{u}\abs{\eta})^2} \tilde{F}_2(u,\eta).
  \end{align*}
  However, for $\Phi(u, \eta) \in R_1$ we have
  \begin{align*}
    \frac{\tilde{F}_2(u, \eta)}{F_1(u,\eta)} =
    (2\pi-\abs{\eta})^{n-\frac{1}{2}} \le (2\pi)^{n-\frac{1}{2}}.
  \end{align*}
  Thus
  \begin{align*}
    \int_{t=t_2}^{t=t_3} p_1(u,t\eta) A(u,t\eta) t^{q}\,dt &\le
    \frac{C'}{(\abs{u}\abs{\eta})^2} {F}_1(u,\eta) \\
    &\le \frac{C''}{(\abs{u}\abs{\eta})^2} p_1(u, \eta) A(u, \eta)
  \end{align*}

  The third integral is more subtle.  We apply Claim
  \ref{single-region} with $\tau_0 = \tau_1 = t_3$, $\tau_3 =
  \frac{2\pi}{\abs{\eta}}$, $R = R_3$, $F = {F}_3$:
  \begin{align*}
    \int_{t=t_3}^{t=\frac{2\pi}{\abs{\eta}}} p_1(u,t\eta) A(u,t\eta)
    t^{q}\,dt &\le C \frac{t_3^{q-1}}{(\abs{u}\abs{\eta})^2} {F}_3(u,
    t_3 \eta)
  \end{align*}
  Then 
  \begin{align}\label{F3-ratio}
    \frac{t_3^{q-1} {F}_3(u, t_3 \eta)}{F_1(u, \eta)} &=
    t_3^{q-1} \abs{u}^{2n-1} \abs{\eta}^{-2n-m}
    (2\pi-t_3 \abs{\eta})^{2n-1} e^{-\frac{1}{4} (\abs{u} \abs{\eta})^2
      (t_3^2 - 1)}.
  \end{align}
  We must show that this ratio is bounded.  Fix some $\epsilon > 0$. If $\abs{u} \ge
  (\pi-\epsilon)^{-1/2} > \pi^{-1/2}$, we have
  $2\pi-\frac{1}{\abs{u}^2} > \pi + \epsilon$ and thus $t_3 =
  \frac{1}{\abs{\eta}}\left(2\pi-\frac{1}{\abs{u}^2}\right)$.  Then
  \begin{align*}
     \abs{\eta}^2  (t_3^2 - 1) &= \left(2\pi-\frac{1}{\abs{u}^2}\right)^2 -
     \abs{\eta}^2 \\
     &\ge (\pi + \epsilon)^2 - \pi^2 = 2 \pi \epsilon + \epsilon^2.
  \end{align*}
  So in this case (\ref{F3-ratio}) becomes
  \begin{align*}
    \frac{t_3^{q-1} {F}_3(u, t_3 \eta)}{F_1(u, \eta)} &=
    \left(\frac{1}{\abs{\eta}}\left(2\pi-\frac{1}{\abs{u}^2}\right)\right)^{q-1}
    \abs{u}^{2n-1} \abs{\eta}^{-2n-m}
    \left(\frac{1}{\abs{u}^2}\right)^{2n-1} e^{-\frac{1}{4} (\abs{u} \abs{\eta})^2
      (t_3^2 - 1)} \\
    &= \left(2\pi-\frac{1}{\abs{u}^2}\right)^{q-1} \abs{u}^{-2n+1}
    \abs{\eta}^{-2n-m-q+1}  e^{-\frac{1}{4} (\abs{u} \abs{\eta})^2
      (t_3^2 - 1)} \\
    &\le (2\pi)^{q-1} \abs{u}^{m+q} e^{-\frac{1}{4}(2 \pi \epsilon + \epsilon^2)\abs{u}^2}
  \end{align*}
  since $\abs{\eta} \le \frac{1}{\abs{u}}$.  This is certainly bounded
  by some constant.  On the other hand, if $\abs{u} \le
  (\pi-\epsilon)^{-1/2}$, then $\abs{\eta} \ge (\pi-\epsilon)^{1/2}$
  and $1 \le t_3 \le \left(\frac{\pi + \epsilon}{\pi -
    \epsilon}\right)^{1/2}$, so that the right side of (\ref{F3-ratio})
  is clearly bounded.

  Thus we have
  \begin{align*}
    \int_{t=t_3}^{t=\frac{2\pi}{\abs{\eta}}} p_1(u,t\eta) A(u,t\eta) t^{q}\,dt &\le
    \frac{C'}{(\abs{u}\abs{\eta})^2} {F}_1(u,\eta) \\
    &\le \frac{C''}{(\abs{u}\abs{\eta})^2} p_1(u, \eta) A(u, \eta)
  \end{align*}
  This completes the proof of this case.
\end{case}

\begin{case}
  Suppose that $\Phi(u,\eta) \in R_2$.  We have
  \begin{equation*}
    \int_{t=1}^{t=\frac{2\pi}{\abs{\eta}}} p_1(u,t\eta) A(u,t\eta)
    t^{q}\,dt = \int_{t=1}^{t=t_3} + \int_{t=t_3}^{t=\frac{2\pi}{\abs{\eta}}}.
  \end{equation*}
  Note that in this region we have $1 \le t_3 \le 2$.
  Again by Claim \ref{single-region}, with $\tau_0 = \tau_1 = 1$ and
  $\tau_2 = t_3$, we have
  \begin{align*}
    \int_{t=1}^{t=t_3} p_1(u,t\eta) A(u,t\eta) t^{q}\,dt 
    \le  \frac{C}{(\abs{u}\abs{\eta})^2} F_2(u,\eta) 
    \le  \frac{C'}{(\abs{u}\abs{\eta})^2} p_1(u,\eta) A(u,\eta).
  \end{align*}
  For the second integral, we apply Claim \ref{single-region} with
  $\tau_0 = 1$, $\tau_1 = t_3$, $\tau_2 = \frac{2\pi}{\abs{\eta}}$ to get
  \begin{align*}
    \int_{t=t_3}^{t=\frac{2\pi}{\abs{\eta}}} p_1(u,t\eta) A(u,t\eta)
    t^{q}\,dt &\le  \frac{C}{(\abs{u}\abs{\eta})^2} F_3(u,
    \eta).
  \end{align*}
  But $\abs{\eta} \ge 2\pi - \frac{1}{\abs{u}^2}$ on $R_3$, so we have
  \begin{align*}
    \frac{F_3(u, \eta)}{F_2(u, \eta)} &= \abs{u}^{2n-1}
      (2\pi-\abs{\eta})^{n-\frac{1}{2}} \\
      &\le \abs{u}^{2n-1}
        \left(\frac{1}{\abs{u}^2}\right)^{n-\frac{1}{2}} = 1.      
  \end{align*}
  Thus
  \begin{align*}
    \int_{t=t_3}^{t=\frac{2\pi}{\abs{\eta}}} p_1(u,t\eta) A(u,t\eta) t^{q}\,dt &\le
    \frac{C'}{(\abs{u}\abs{\eta})^2} {F}_2(u,\eta) \\
    &\le \frac{C''}{(\abs{u}\abs{\eta})^2} p_1(u, \eta) A(u, \eta).
  \end{align*}
\end{case}

\begin{case}
  Suppose $\Phi(u,\eta) \in R_3$; we apply Claim \ref{single-region}
  with $\tau_0 = \tau_1 = 1$, $\tau_2 = \frac{2\pi}{\abs{\eta}}$ to
  get
  \begin{equation*}
    \int_{t=1}^{t=\frac{2\pi}{\abs{\eta}}} p_1(u,t\eta) A(u,t\eta) t^{q}\,dt \le C
    \frac{1}{(\abs{u}\abs{\eta})^2} F_3(u,\eta) \le C' p_1(u,\eta) A(u,\eta).
  \end{equation*}
\end{case}

The three cases together complete the proof of Lemma \ref{geodesic-integral-bound}.
\end{proof}

\begin{notation}
   For $f \in C^1(G)$, let $m_f := \frac{\int_B
     f\,d\haar}{\int_B d\haar}$, where $B$ is the Carnot-Carath\'eodory unit ball.  .
\end{notation}

To continue to follow the line of \cite{bbbc-jfa}, we need the
following Poincar\'e inequality.  This theorem can be found in
\cite{jerison}, and is a special case of a more general theorem
appearing in \cite{maheux-saloffcoste}.

\begin{theorem}\label{poincare}
  There exists a constant $C$ such that for any $f \in C^\infty(G)$,
  \begin{equation}\label{poincare-eq}
    \int_B \abs{f - m_f}\,d\haar \le C \int_B \abs{\grad f}\,d\haar.
  \end{equation}
\end{theorem}

\begin{corollary}\label{poincare-p_1}
    There exists a constant $C$ such that for any $f \in C^\infty(G)$,
  \begin{equation}\label{poincare-p_1-eq}
    \int_B \abs{f - m_f} p_1 \,d\haar \le C \int_B \abs{\grad f}p_1 \,d\haar.
  \end{equation}
\end{corollary}

%

\begin{proof}
  $p_1$ is bounded and bounded below away from $0$ on $B$.
\end{proof}

\begin{lemma}[akin to Lemma 5.2 of \cite{bbbc-jfa}]\label{bbbc-52}
  There exists a constant $C$ such that for all $f \in \funcclass$,
  \begin{equation}\label{bbbc-52-eq}
    \int_{B^C} \abs{f-m_f} p_1\,d\haar \le C \int_G \abs{\grad f} p_1\,d\haar.
  \end{equation}
\end{lemma}

\begin{proof}
  Changing to $(u,\eta)$ coordinates, we wish to show
  \begin{equation}
    \int_{\abs{u} \ge \frac{1}{2\pi}} \int_{\frac{1}{\abs{u}} \le
      \abs{\eta} < 2\pi} \abs{f(\Phi(u,\eta))-m_f}p_1(\Phi(u,\eta))
    A(u,\eta)\,d\eta\,du \le C \int_{G} \abs{\grad f} p_1 \,d\haar.
  \end{equation}
  (The limits of integration are as described in (\ref{BCPhi}).)  By
  an abuse of notation we shall write $f(u, \eta)$ for
  $f(\Phi(u,\eta))$, $p_1(u,\eta)$ for $p_1(\Phi(u,\eta))$, $\grad
  f(u,\eta)$ for $(\grad f)(\Phi(u,\eta))$, et cetera.

  Let $g(u,\eta) := f\left(u,
  \min\left(\abs{\eta},\frac{1}{\abs{u}}\right)
  \frac{\eta}{\abs{\eta}}\right)$.  Then $g=f$ on $B$ (in particular
  $m_g = m_f$), $g$ is bounded, the function $s \mapsto g(u, s\eta)$
  is absolutely continuous, and $\frac{d}{ds}g(u,s\eta) = 0$ for $s >
  \frac{1}{\abs{u}\abs{\eta}}$.

  Now $\abs{f - m_f} \le \abs{f-g} + \abs{g - m_f}$.  We first
  observe that for $\abs{u}\abs{\eta} \ge 1$ we have
  \begin{align*}
    \abs{f(u,\eta) - g(u,\eta)} &=
    \abs{\int_{s=\frac{1}{\abs{u}\abs{\eta}}}^{s=1} \left(\frac{d}{ds}
      f(u,s\eta)-\cancel{\frac{d}{ds}g(u,s\eta)}\right)\,ds} \\
    &\le \int_{s=\frac{1}{\abs{u}\abs{\eta}}}^{s=1} \abs{\grad
      f(u,s\eta)} \abs{u} \abs{\eta}\,ds
  \end{align*} 
  by (\ref{geodesic-grad}).  Thus
  \begin{align*}
    \int_{B^C} \abs{f-g}p_1\,d\haar &= \int_{\abs{u} \ge \frac{1}{2\pi}} \int_{\abs{\eta} \ge
      \frac{1}{\abs{u}}} \abs{f(u,\eta)-g(u,\eta)}p_1(u,\eta)
    A(u,\eta)\,d\eta\,du
\\
\intertext{where the limits of integration come from the conditions
  $\abs{u} \abs{\eta} \ge 1$, $\abs{\eta} < 2\pi$;}
 &\le \int_{\abs{u} \ge \frac{1}{2\pi}} \int_{\abs{\eta} \ge
      \frac{1}{\abs{u}}} \int_{s=\frac{1}{\abs{u}\abs{\eta}}}^{s=1} \abs{\grad
      f(u,s\eta)} \abs{u} \abs{\eta}p_1(u,\eta)
    A(u,\eta)\,ds\,d\eta\,du \\
 &= \int_{\abs{u} \ge \frac{1}{2\pi}} \int_{s=0}^{s=1}
    \int_{\frac{1}{s\abs{u}} \le \abs{\eta} \le 2\pi} \abs{\grad
      f(u,s\eta)} \abs{u} \abs{\eta}p_1(u,\eta)
    A(u,\eta)\,d\eta\,ds\,du
\intertext{by Tonelli's theorem.  We now make the change of variables $\eta' =
  s\eta$ to obtain}
&= \int_{\abs{u} \ge \frac{1}{2\pi}} \int_{s=0}^{s=1}
    \int_{\frac{1}{\abs{u}} \le \abs{\eta'} \le 2\pi s} \abs{\grad
      f(u,\eta')} \abs{u} \frac{1}{s}\abs{\eta'}p_1\left(u,\frac{1}{s}\eta'\right)
    A\left(u,\frac{1}{s}\eta'\right) \frac{1}{s^m}\,d\eta'\,ds\,du \\
    &= \int_{\abs{u} \ge \frac{1}{2\pi}} \int_{\frac{1}{\abs{u}} \le \abs{\eta'} \le 2\pi} \abs{\grad
      f(u,\eta')} \abs{u} \abs{\eta'} \\
    &\quad \times \left(\int_{s=\frac{\abs{\eta'}}{2\pi}}^{s=1}
     p_1\left(u,\frac{1}{s}\eta'\right)
    A\left(u,\frac{1}{s}\eta'\right) \frac{1}{s^{m+1}}\,ds\right)\,d\eta'\,du
\intertext{Make the further change of variables $t=\frac{1}{s}$ to
  get}
    &= \int_{\abs{u} \ge \frac{1}{2\pi}} \int_{\frac{1}{\abs{u}} \le \abs{\eta'} \le 2\pi} \abs{\grad
      f(u,\eta')} \abs{u} \abs{\eta'} \left(\int_{t=1}^{t=\frac{2\pi}{\abs{\eta'}}}
     p_1(u,t\eta')
    A(u,t\eta') t^{m-1}\,dt\right)\,d\eta'\,du.
    \intertext{Applying Lemma \ref{geodesic-integral-bound} to the
      bracketed term gives}
    &\le C \int_{\abs{u} \ge \frac{1}{2\pi}} \int_{\frac{1}{\abs{u}} \le
      \abs{\eta'} \le 2\pi} \frac{1}{\abs{u} \abs{\eta'}} \abs{\grad
      f(u,\eta')} p_1(u,\eta') A(u,\eta') \,d\eta' \,du  \\
    &\le C' \int_{B^C} \abs{\grad f} p_1\,d\haar
  \end{align*}
  converting back from geodesic coordinates and using the fact that
  $\abs{u}\abs{\eta'} \ge 1$.
  
  To complete the proof, we must show that $\int_{B^C} \abs{g - m_f}
  p_1\,d\haar \le \int_G \abs{\grad f}p_1\,d\haar$.  Note that
  for $\Phi(u,\eta) \in B^C$, i.e. $\abs{u} \abs{\eta} \ge 1$, we
  have $g(u,\eta) = f\left(u, \frac{1}{\abs{u}
    \abs{\eta}}\eta\right)$, so
  \begin{align}
    \int_{B^C} \abs{g - m_f} p_1\,d\haar &= \int_{\abs{u} \ge \frac{1}{2\pi}}
    \int_{\frac{1}{\abs{u}} \le \abs{\eta} \le 2\pi } \abs{f\left(u, \frac{1}{\abs{u}
    \abs{\eta}}\eta\right) - m_f} p_1(u,\eta) A(u,\eta) \,d\eta\,du.
    \intertext{Change the $\eta$ integral to polar coordinates by
      writing $\eta = \rho \hat{\eta}$, where $\rho \ge 0$ and
      $\abs{\hat{\eta}} = 1$.  Note that $p_1(u,\eta),
      A(u,\eta)$ depend on $\eta$ only through $\rho$ and not $\hat{\eta}$.}
&= C \int_{\abs{u} \ge \frac{1}{2\pi}}
    \int_{\hat{\eta} \in S^{m-1}} \abs{f\left(u, \frac{1}{\abs{u}}
      \hat{\eta}\right) - m_f}
    \int_{\rho=\frac{1}{\abs{u}}}^{\rho=2\pi} p_1(u,\rho) A(u,\rho)
    \rho^{m-1}\,d\rho \,d\hat{\eta}\,du \label{foogle}
  \end{align}
  Now, for any $s \in [0,1]$ we have
  \begin{equation}\label{gmf1}
    \abs{f\left(u, \frac{1}{\abs{u}}\hat{\eta}\right) - m_f} \le
    \abs{f\left(u, \frac{1}{\abs{u}}\hat{\eta}\right) - f\left(u,
      \frac{s}{\abs{u}}\hat{\eta}\right)} + \abs{ f\left(u,
      \frac{s}{\abs{u}}\hat{\eta}\right) - m_f}.
  \end{equation}
  Let 
\begin{equation}
D(u) := \int_{s=0}^{s=1} \frac{s^{m-1}}{\abs{u}^m} A\left(u,
  \frac{s}{\abs{u}}\right)\,ds.
\end{equation}
By multiplying both sides of (\ref{gmf1}) by
  $\frac{1}{D(u)}\frac{s^{m-1}}{\abs{u}^m} A\left(u,
  \frac{s}{\abs{u}}\right)$ and integrating we obtain
  \begin{equation}\label{xyzzy}
    \begin{split}
    \abs{f\left(u, \frac{1}{\abs{u}}\hat{\eta}\right) - m_f} &\le \frac{1}{D(u)} \int_{s=0}^{s=1} \left(\abs{f\left(u, \frac{1}{\abs{u}}\hat{\eta}\right) - f\left(u,
      \frac{s}{\abs{u}}\hat{\eta}\right)} + \abs{ f\left(u,
      \frac{s}{\abs{u}}\hat{\eta}\right) - m_f}\right) \\
    &\quad \times \frac{s^{m-1}}{\abs{u}^m} A\left(u,
  \frac{s}{\abs{u}}\right)\,ds.
  \end{split}
  \end{equation}
  Let
  \begin{equation}\label{R-def}
    R(u) := \frac{1}{D(u)} \int_{\rho=\frac{1}{\abs{u}}}^{\rho=2\pi}
      p_1(u,\rho) A(u,\rho) \rho^{m-1}\,d\rho.
  \end{equation}
  Then substituting (\ref{xyzzy}) into (\ref{foogle}) and using
  (\ref{R-def}) we have
  \begin{equation}
    \int_{B^C} \abs{g-m_f}p_1\,d\haar \le I_1 + I_2
  \end{equation}
  where
  \begin{align}
     I_1 &:= \int_{\abs{u} \ge \frac{1}{2\pi}}
    \int_{\hat{\eta} \in S^{m-1}} 
\int_{s=0}^{s=1} \abs{f\left(u, \frac{1}{\abs{u}}\hat{\eta}\right) - f\left(u, \frac{s}{\abs{u}}\hat{\eta}\right)} \frac{s^{m-1}}{\abs{u}^m} A\left(u,
\frac{s}{\abs{u}}\right)\,ds \,
   R(u) 
   \,d\hat{\eta}\,du \label{I1-def} \\
     I_2 &:= \int_{\abs{u} \ge \frac{1}{2\pi}}
    \int_{\hat{\eta} \in S^{m-1}} 
\int_{s=0}^{s=1} \abs{ f\left(u, \frac{s}{\abs{u}}\hat{\eta}\right) -
  m_f} \frac{s^{m-1}}{\abs{u}^m} A\left(u,
\frac{s}{\abs{u}}\right)\,ds\, R(u) \,d\hat{\eta}\,du. \label{I2-def}
  \end{align}
  We now show that $I_1$, $I_2$ can each be bounded by a constant
  times $\int_G \abs{\grad f}p_1\,d\haar$, using the
  following claim.
  \begin{claim}\label{R-claim}
    There exists a constant $C$ such that for all $\abs{u} \ge
    \frac{1}{2\pi}$ we have
    \begin{equation}
      R(u) \le C
      \left(2\pi-\frac{1}{\abs{u}}\right)^{2n-1} \le (2\pi)^{2n-1}C.
    \end{equation}
  \end{claim}

  \begin{proof}[Proof of Claim]
    First, by Corollary \ref{A-estimates} we have
    \begin{align*}
      D(u) &:= \int_{s=0}^{s=1} \frac{s^{m-1}}{\abs{u}^m} A\left(u,
  \frac{s}{\abs{u}}\right)\,ds \\
  &\ge C \int_{s=0}^{s=1} \frac{s^{m-1}}{\abs{u}^m} \abs{u}^{2m}
  \left(\frac{s}{\abs{u}}\right)^{2(m+n)} \left(2\pi -
  \frac{s}{\abs{u}}\right)^{2n-1}\,ds \\
  &= C \abs{u}^{-2n-m} \int_{s=0}^{s=1} s^{3m+2n-1} \left(2\pi -
  \frac{s}{\abs{u}}\right)^{2n-1}\,ds \\
  &\ge C \abs{u}^{-2n-m}  \int_{s=0}^{s=1} s^{3m+2n-1}
  \left(2\pi(1-s)\right)^{2n-1}\,ds && \text{since $u \ge
    \frac{1}{2\pi}$} \\
  &= C' \abs{u}^{-2n-m}
    \end{align*}
    since the $s$ integral is a positive constant independent of $u$.

    On the other hand, making the change of variables $\rho =
    \frac{t}{\abs{u}}$ shows
    \begin{align*}
      \int_{\rho=\frac{1}{\abs{u}}}^{\rho=2\pi}
      p_1(u,\rho) A(u,\rho) \rho^{m-1}\,d\rho &= \abs{u}^{-m} \int_{t=1}^{t=2\pi\abs{u}}
      p_1\left(u,\frac{t}{\abs{u}}\right) A\left(u,\frac{t}{\abs{u}}\right) t^{m-1}\,dt \\
      &\le C \abs{u}^{-m} p_1\left(u, \frac{1}{\abs{u}}\right) A\left(u, \frac{1}{\abs{u}}\right)
    \end{align*}
    by taking $\abs{\eta} = \frac{1}{\abs{u}}$ in Lemma
    \ref{geodesic-integral-bound}.  Now $p_1\left(u, \frac{1}{\abs{u}}\right)$ is the heat
    kernel evaluated at a point on the unit sphere of $G$, so this is
    bounded by a constant independent of $u$.  Thus by Corollary
    \ref{A-estimates} we have
    \begin{align*}
       \int_{\rho=\frac{1}{\abs{u}}}^{\rho=2\pi}
      p_1(u,\rho) A(u,\rho) \rho^{m-1}\,d\rho &\le C \abs{u}^{-m}
      \abs{u}^{2m} \left(\frac{1}{\abs{u}}\right)^{2(m+n)} \left(2\pi-\frac{1}{\abs{u}}\right)^{2n-1} \\
      &\le C \left(2\pi-\frac{1}{\abs{u}}\right)^{2n-1} \abs{u}^{-2n-m}.
    \end{align*}
    Combining this with the estimate on $D(u)$ proves the claim.
  \end{proof}

  To estimate $I_1$ (see (\ref{I1-def})), we observe that
  \begin{align*}
    \abs{f\left(u, \frac{1}{\abs{u}}\hat{\eta}\right) - f\left(u,
      \frac{s}{\abs{u}}\hat{\eta}\right)} &= \abs{\int_{t=s}^{t=1} \frac{d}{dt}
      f\left(u, \frac{t}{\abs{u}}\hat{\eta}\right)\,dt} \\
    &\le \int_{t=s}^{t=1} \abs{\frac{d}{dt}
      f\left(u, \frac{t}{\abs{u}}\hat{\eta}\right)}\,dt \\
    &\le \int_{t=s}^{t=1} \abs{\grad f\left(u,
        \frac{t}{\abs{u}}\hat{\eta}\right)}\,dt
  \end{align*}
  by (\ref{geodesic-grad}).  Thus
  \begin{align}
    I_1 &\le \int_{\abs{u} \ge \frac{1}{2\pi}} \int_{\hat{\eta} \in
      S^{m-1}} \int_{s=0}^{s=1} \int_{t=s}^{t=1} \abs{\grad f\left(u,
      \frac{t}{\abs{u}}\hat{\eta}\right)} \frac{s^{m-1}}{\abs{u}^m} A\left(u,
    \frac{s}{\abs{u}}\right)\,dt\,ds \, R(u) \,d\hat{\eta}\,du \\
    &= \int_{\abs{u} \ge \frac{1}{2\pi}} \int_{\hat{\eta} \in
      S^{m-1}} \int_{t=0}^{t=1} \abs{\grad f\left(u,
      \frac{t}{\abs{u}}\hat{\eta}\right)}
    \frac{1}{\abs{u}^m}\left(R(u)\int_{s=0}^{s=t}  s^{m-1} A\left(u, 
    \frac{s}{\abs{u}}\right)\,ds\right)\,dt \, d\hat{\eta}\,du.
  \end{align}
  Now by Claim \ref{R-claim} and Corollary \ref{A-estimates}, we have
  for all $t \in [0,1]$: 
  \begin{align*}
    R(u) \int_{s=0}^{s=t}  s^{m-1} A\left(u,
    \frac{s}{\abs{u}}\right)\,ds &\le C 
      \left(2\pi-\frac{1}{\abs{u}}\right)^{2n-1} \int_{s=0}^{s=t}
      s^{m-1} \abs{u}^{2m} \left(\frac{s}{\abs{u}}\right)^{2(m+n)}
      \left(2\pi-\frac{s}{\abs{u}}\right)^{2n-1}\,ds \\
      &\le C 
      \left(2\pi-\frac{t}{\abs{u}}\right)^{2n-1} (2\pi)^{2n-1}
      \abs{u}^{-2n} \int_{s=0}^{s=t}
      s^{3m+2n-1}\,ds \\
      &= C' \left(2\pi-\frac{t}{\abs{u}}\right)^{2n-1} \abs{u}^{-2n}
      t^{3m+2n} \\
      &= C' \left(2\pi-\frac{t}{\abs{u}}\right)^{2n-1} \abs{u}^{2m}
      \left(\frac{t}{\abs{u}}\right)^{2(m+n)} t^m \\
      &\le C'' A\left(u, \frac{t}{\abs{u}}\right) t^m \\
      &\le C'' A\left(u, \frac{t}{\abs{u}}\right) t^{m-1}.
  \end{align*}
  Thus
  \begin{align}
    I_1 &\le C \int_{\abs{u} \ge \frac{1}{2\pi}} \int_{\hat{\eta} \in
      S^{m-1}} \int_{t=0}^{t=1} \abs{\grad f\left(u,
      \frac{t}{\abs{u}}\hat{\eta}\right)} A\left(u, \frac{t}{\abs{u}}\right)
    \frac{t^{m-1}}{\abs{u}^m}\,dt \, d\hat{\eta}\,du \\
    \intertext{Make the change of variables $r=\frac{t}{\abs{u}}$:}
    &= C \int_{\abs{u} \ge \frac{1}{2\pi}} \int_{\hat{\eta} \in
      S^{m-1}} \int_{r=0}^{r=\frac{1}{\abs{u}}} \abs{\grad f\left(u,
      r\hat{\eta}\right)} A\left(u, r\right) r^{m-1} \,dr \,
    d\hat{\eta}\,du \\
    &\le C \int_{u \in \R^{2n}} \int_{\hat{\eta} \in
      S^{m-1}} \int_{r=0}^{r=\frac{1}{\abs{u}}} \abs{\grad f\left(u,
      r\hat{\eta}\right)} A\left(u, r\right) r^{m-1} \,dr \,
    d\hat{\eta}\,du \\
    &= C \int_{B}  \abs{\grad f}\,d\haar \label{int-B-gradf} \\
    &\le \frac{C}{\inf_B p_1} \int_B \abs{\grad f} p_1\,d\haar \\
    &\le C' \int_G  \abs{\grad f} p_1\,d\haar. \label{int-G-gradf-p_1}
  \end{align}
  where we have used the fact that $p_1$ is bounded away from $0$ on
  $B$.

  For $I_2$ (see (\ref{I2-def})), we have by Claim \ref{R-claim} that
  \begin{align}
    I_2 &\le C \int_{\abs{u} \ge \frac{1}{2\pi}}
    \int_{\hat{\eta} \in S^{m-1}} 
\int_{s=0}^{s=1} \abs{ f\left(u, \frac{s}{\abs{u}}
        \hat{\eta}\right) - m_f} \frac{s^{m-1}}{\abs{u}^m} A\left(u,
  \frac{s}{\abs{u}}\right)\,ds
    \,d\hat{\eta}\,du.
\intertext{Make the change of variables $r=\frac{s}{\abs{u}}$:}
&= C \int_{\abs{u} \ge \frac{1}{2\pi}}
    \int_{\hat{\eta} \in S^{m-1}} 
\int_{r=0}^{r=\frac{1}{\abs{u}}} \abs{ f\left(u, r\hat{\eta}\right) - m_f} r^{m-1} A\left(u,
  r\right)\,dr
    \,d\hat{\eta}\,du \\
    &\le C \int_{u \in \R^{2n}} \int_{\hat{\eta} \in S^{m-1}} 
    \int_{r=0}^{r=\frac{1}{\abs{u}}} 
    \abs{ f\left(u, r\hat{\eta}\right) - m_f} r^{m-1}
    A\left(u, r\right)\,dr\,d\hat{\eta} du \\
    &= C \int_B \abs{f-m_f}\,d\haar \\
    &\le C \int_B \abs{\grad f}\,d\haar
  \end{align}
  by Theorem \ref{poincare}.  The inequalities
  (\ref{int-B-gradf}--\ref{int-G-gradf-p_1}) now show that $I_2
  \le C' \int_G \abs{\grad f}p_1\,d\haar$, as desired.
\end{proof}


\begin{corollary}\label{cheeger-combined}
  There exists a constant $C$ such that for all $f \in \funcclass$,
  \begin{equation}
    \int_G \abs{f - m_f} p_1 \,d\haar \le C \int_G \abs{\grad f} p_1 \,d\haar.
  \end{equation}
\end{corollary}
\begin{proof}
  Add (\ref{poincare-p_1-eq}) and (\ref{bbbc-52-eq}).
\end{proof}

We can now prove some cases of the desired gradient inequality
(\ref{integral-to-show}).

\begin{notation}
  Let $\cylinder(R) = \{(x,z) : \abs{x} \le R\}$ denote the ``cylinder
  about the $z$ axis'' of radius $R$.
\end{notation}

\begin{lemma}\label{grad-ineq-cyl}
  For fixed $R > 0$, (\ref{integral-to-show}) holds, with a constant
  $C=C(R)$ depending on $R$, for all $f \in \funcclass$ which are
  supported on $\cylinder(R)$ and satisfy $m_f = 0$.
\end{lemma}
\begin{proof}
  \begin{align*}
     \abs{\int_G ((\grad-\hat{\grad})f) p_1}\,d\haar &= \abs{\int_G f
       (\grad-\hat{\grad})p_1}\,d\haar &&\text{by integration by parts (\ref{byparts})} \\
     &\le \int_G \abs{f}
     \abs{(\grad-\hat{\grad})p_1}\,d\haar \\
     &= \int_G \abs{f} \abs{x} \abs{\grad_z p_1} \,d\haar &&
     \text{by (\ref{abs-grad-difference})}\\
       &\le C R \int_G \abs{f} p_1 \,d\haar && \text{by
         (\ref{gradzp1-estimate}); note $\abs{x} \le R$ on the support
       of $f$} \\
       &\le C' R \int_G \abs{\grad f} p_1 \,d\haar && \text{by Corollary \ref{cheeger-combined}}.
  \end{align*}
\end{proof}

\begin{notation}
  If $T : G \to M_{2n \times 2n}$ is a matrix-valued function on $G$,
  with $k \ell$th entry $a_{k \ell}$, let $\nabla \cdot T : G \to
  \R^{2n}$ be defined as
  \begin{equation}
    \nabla \cdot T(g) := \sum_{k,\ell=1}^{2n} X_\ell a_{k \ell}(g) e_k.
  \end{equation}
  Note that for $f : G \to \R$ we have the product formula
  \begin{equation}\label{div-product}
    \nabla \cdot (f T) = T \grad f + f \nabla \cdot T.
  \end{equation}
  \end{notation}

\begin{lemma}\label{grad-ineq-offcyl}
  For fixed $R > 1$, (\ref{integral-to-show}) holds, with a constant
  $C=C(R)$ depending on $R$, for all $f \in \funcclass$ which are
  supported on the complement of $\cylinder(R)$.
\end{lemma}

\begin{proof}
  Applying (\ref{gradient-combos}) we have
  \begin{equation*}
    \grad p_1(x,z) = \grad_x p_1(x,z) + \frac{1}{2}J_{\grad_z p_1(x,z)}x.
  \end{equation*}
  Now $p_1$ is a ``radial'' function (that is, $p_1(x,z)$ depends
  only on $\abs{x}$ and $\abs{z}$).  Thus we have that $\grad_x
  p_1(x,z)$ is a scalar multiple of $x$, and also that $\grad_z
  p_1(x,z)$ is a scalar multiple of $z$, so that $J_{\grad_z
    p_1(x,z)}x$ is a scalar multiple of $J_z x$.  

  For nonzero $x \in \R^{2n}$, let $T(x) \in M_{2n \times 2n}$ be
  orthogonal projection onto the $m$-dimensional subspace of $\R^{2n}$
  spanned by the orthogonal vectors $J_{u_1}x, \dots, J_{u_m}x$.
  (Recall $\inner{J_{u_i}x}{J_{u_j} x} = -\inner{u_i}{u_j} \norm{x}^2
  = -\delta_{ij} \norm{x}^2$.)  Thus for any $z \in \R^m$, $T(x) J_z x
  = J_z x$, and $T(x) x = 0$; in particular,
  \begin{equation}\label{Tgradp}
    T(x) \grad p_1(x,z) =
  \frac{1}{2}J_{\grad_z p_1(x,z)}x = \frac{1}{2}
  (\grad-\hat{\grad})p_1(x,z).
  \end{equation}

  Explicitly, we have
  \begin{equation*}
    T(x) = \frac{1}{\abs{x}^2} \sum_{j=1}^m J_{u_j} x (J_{u_j}x)^{T}.
  \end{equation*}
   Note that $\abs{T(x)} = 1$ (in operator norm) for all $x \ne 0$,
   and a routine computation verifies that $\abs{\nabla \cdot T(x)} =
   \abs{\nabla_x \cdot T(x)} \le \frac{C}{\abs{x}}$.  Indeed, the
   $k\ell$th entry of $T(x)$ is
\begin{equation*}
a_{k\ell}(x) = \frac{1}{\abs{x}^2} \sum_{j=1}^{m}
\inner{J_{u_j}x}{e_k}\inner{J_{u_j}x}{e_\ell}
\end{equation*}
so that $\abs{X_k a_{k\ell}(x)} =
\abs{\frac{\partial}{\partial x^k} a_{k\ell}(x)} \le
\frac{3m}{\abs{x}}$; thus $\abs{\nabla \cdot T(x)} \le
\frac{3m(2n)^2}{\abs{x}}$.

  Since $p_1$ decays rapidly at infinity, we have the integration by
  parts formula
  \begin{equation}\label{div-parts}
    0 = \int_G \nabla \cdot (f p_1T)\,d\haar = \int_G (f p_1 \nabla \cdot T +
    f T \grad p_1 + p_1 T \grad f) \,d\haar.
  \end{equation}
  Thus
  \begin{align*}
    \abs{\int_G ((\grad - \hat{\grad})f) p_1\,d\haar} &= \abs{\int_G f
      (\grad - \hat{\grad})p_1\,d\haar} \\
    &= 2\abs{\int_G f T \grad p_1\,d\haar} \\
    &= 2\abs{\int_G f p_1 (\nabla \cdot T +  T \grad f)\,d\haar} \\
    &\le 2\int_G \abs{f} \abs{\nabla \cdot T} p_1\,d\haar + 2 \int_G
    \abs{T} \abs{\grad f} p_1\,d\haar\\
    &\le \frac{2C}{R} \int_G \abs{f} p_1\,d\haar + 2 \int_G
    \abs{\grad f} p_1\,d\haar
  \end{align*}
  since on the support of $f$, we have $ \abs{\nabla \cdot T} \le
  \frac{C}{\abs{x}} \le \frac{C}{R}$, and $\abs{T} = 1$.  The second integral is the
  desired right side of (\ref{integral-to-show}).  The first integral
  is bounded by the same by Corollary \ref{cheeger-combined}, where we
  note that $m_f = 0$ because $f$ vanishes on $D(R) \supset B$.
\end{proof}

We can now complete the proof of Theorem \ref{main-grad-theorem}.

\begin{proof}[Proof of Theorem \ref{main-grad-theorem}]
  We prove (\ref{integral-to-show}) for general $f \in \funcclass$.
  By replacing $f$ by $f - m_f \in \funcclass$, we can assume $m_f = 0$.

  Let $\psi \in C^\infty(G)$ be a smooth function such that $\psi
  \equiv 1$ on $D(1)$ and 
  $\psi$ is supported in $D(2)$.  Then $f = \psi f + (1-\psi)f$.

  $\psi f$ is supported on $D(2)$, so Lemma \ref{grad-ineq-cyl}
  applies to $\psi f$.  (Note that $m_{\psi f} = 0$ since $\psi \equiv
  1$ on $D(1) \supset B$.)  We have
  \begin{align*}
    \abs{\int_G (\grad - \hat{\grad})(\psi f) p_1\,d\haar} &\le C \int_G
    \abs{\grad(\psi f)} p_1\,d\haar \\
   &\le C \int_G \abs{\grad \psi} \abs{f} p_1 \,d\haar + \int_G \abs{\psi}
    \abs{\grad f} p_1\,d\haar \\
    &\le C \sup_G \abs{\grad \psi} \int_G \abs{f} p_1 \,d\haar + C \sup_G
    \abs{\psi}\int_G \abs{\grad f} p_1\,d\haar.
  \end{align*}
  The second integral is the right side of (\ref{integral-to-show}),
  and the first is bounded by the same by Corollary
  \ref{cheeger-combined}.

  Precisely the same argument applies to $(1-\psi)f$, which is
  supported on the complement of $D(1)$, by using Lemma
  \ref{grad-ineq-offcyl} instead of Lemma \ref{grad-ineq-cyl}.
\end{proof}

\section{The optimal constant $K$}

We observed previously that the constant $K$ in (\ref{grad-ineq}) can
be taken to be independent of $t$.  We now show that the
\emph{optimal} constant is also independent of $t > 0$, and is
discontinuous at $t=0$.  This distinguishes the current situation from
the elliptic case, in which the constant is continuous at $t=0$; see,
for instance \cite[Proposition 2.3]{bakry-sobolev}.  This
fact was initially noted for the Heisenberg group $\heis$ in
\cite{driver-melcher}, and the proof here is similar to the one found
there.

\newcommand{\Kopt}{K_{\mathrm{opt}}}

\begin{proposition}
  For $t \ge 0$, let
  \begin{equation}\label{Kopt}
    \Kopt(t) := \sup \left\{ 
    \frac{\abs{(\grad P_t f)(g)}}{P_t (\abs{\grad f})(g)} : f \in \funcclass, g \in G , P_t (\abs{\grad
        f})(g) \ne 0 \right\}
  \end{equation}
  Then $\Kopt(0) = 1$, and for all $t > 0$, $\Kopt(t) \equiv \Kopt >
  1$ is independent of $t$, so that $\Kopt(t)$ is discontinuous at
  $t=0$.  In particular, \mbox{$\Kopt \ge \sqrt{\frac{3n+5}{3n+1}}$}.
\end{proposition}

\begin{proof}
  It is obvious that $\Kopt(0) = 1$.
  
  As before, by the left invariance of $P_t$ and $\grad$, it suffices
  to take $g=0$ on the right side of (\ref{Kopt}).  To show
  independence of $t > 0$, fix $t, s > 0$.  If $f \in \funcclass$,
  then $\tilde{f} := f \circ \varphi_{s^{1/2}}^{-1} \in \funcclass$
  and $f = \tilde{f} \circ \varphi_{s^{1/2}}$.  Then
  \begin{align*}
    \frac{\abs{(\grad P_t f)(0)}}{P_t (\abs{\grad f})(0)} &= 
    \frac{\abs{(\grad P_t (\tilde{f} \circ
        \varphi_{s^{1/2}}))(0)}}{P_t \left(\abs{\grad (\tilde{f} \circ
        \varphi_{s^{1/2}})}\right)(0)} \\
    &=     \frac{\abs{(\grad (P_{st} \tilde{f}) \circ
        \varphi_{s^{1/2}})(0)}}{P_t \left(s^{1/2} \abs{\grad \tilde{f}} \circ
        \varphi_{s^{1/2}}\right)(0)} \\
    &=     \frac{s^{1/2}\abs{(\grad P_{st}
        \tilde{f})(\varphi_{s^{1/2}}(0))}}
         {s^{1/2} P_{st} \left(\abs{\grad
             \tilde{f}}\right)(\varphi_{s^{1/2}}(0))}
         \le \Kopt(st).
  \end{align*}
  Taking the supremum over $f$ shows that $\Kopt(t) \le \Kopt(st)$.
  $s$ was arbitrary, so $\Kopt(t)$ is constant for $t > 0$.  

  In order to bound the constant, we explicitly compute a related ratio for
  a particular choice of function $f$.  The function used is an obvious
  generalization of the example used in \cite{driver-melcher} for the
  Heisenberg group $\heis$.

  Fix a unit vector $u_1$ in the center of $G$, i.e. $u_1 \in 0 \times
  \R^m \subset \R^{2n+m}$.  We note that the operator $L$ and the norm
  of the gradient $\abs{\grad f}^2 = \frac{1}{2}(L(f^2) - 2 f L f)$
  are independent of the orthonormal basis $\{e_i\}$ chosen to define
  the vector fields $\{X_i\}$, so without loss of generality we
  suppose that $J_{u_1} e_1 = e_2$.  Then take
  \begin{align*}
    f(x,z) &:= \inner{x}{e_1} + \inner{z}{u_1}\inner{x}{e_2} = x^1 +
    z^1 x^2 \\
    k(t) &:= \frac{\abs{(\grad P_t f)(0)}}{P_t(\abs{\grad f})(0)}.
  \end{align*}
  Note that $k(t) \le \Kopt$ for all $t$.  By the Cauchy-Schwarz
  inequality,
  \begin{equation*}
    k(t)^2 \ge k_2(t) := \frac{\abs{(\grad P_t f)(0)}^2}{P_t\left(\abs{\grad f}^2\right)(0)}.
  \end{equation*}
  Since $f$ is a polynomial, we can compute $P_t f$ by the formula
  $P_t f = f + \frac{t}{1!} L f + \frac{t^2}{2!} L^2 f + \cdots$ since
  the sum terminates after a finite number of terms (specifically,
  two).  The same is true of $\abs{\grad f}^2$, which is also a
  polynomial (three terms are needed).  The formulas
  (\ref{Xi-formula}) are helpful in carrying out this tedious but
  straightforward computation.
  We find
  \begin{equation*}
    k_2(t) = \frac{(1+t)^2}{1-2t+(3n+2)t^2}
  \end{equation*}
  which, by differentiation, is maximized at $t_{\mathrm{max}} =
  \frac{2}{3n+3}$, with $k_2(t_{\mathrm{max}}) = \frac{3n+5}{3n+1}$.
  Since $\Kopt \ge k(t_{\mathrm{max}}) \ge
  \sqrt{k_2(t_{\mathrm{max}})} = \sqrt{\frac{3n+5}{3n+1}}$, this is
  the desired bound.
\end{proof}

\section{Consequences and possible extensions}

Section 6 of \cite{bbbc-jfa} gives several important consequences of
the gradient inequality (\ref{intro-grad-ineq}).  The proofs given
there are generic (see their Remark 6.6); with Theorem
\ref{main-grad-theorem} in hand, they go through without change in the
case of H-type groups.  These consequences include:
\begin{itemize}
\item Local Gross-Poincar\'e inequalities, or $\varphi$-Sobolev
  inequalities;
\item Cheeger type inequalities; and
\item Bobkov type isoperimetric inequalities.
\end{itemize}
We refer the reader to \cite{bbbc-jfa} for the statements and proofs of
these theorems, and many references as well.




  

It would be very useful to extend the gradient inequality
(\ref{intro-grad-ineq}) to a more general class of groups, such as the
nilpotent Lie groups.  However, this is likely to require a proof
which is divorced from the heat kernel estimates
(\ref{p1-estimate}--\ref{gradzp1-estimate}).  Such precise estimates
are currently not known to hold in more general settings, and could be
difficult to obtain.  A key difficulty is the lack of a convenient
explicit heat kernel formula like (\ref{pt-formula}).

The author would like to express his sincere thanks to his advisor,
Bruce Driver, for a great many helpful discussions during the
preparation of this article.  The author would also like to thank the
anonymous referee for several helpful corrections and comments,
notably the suggestion to include Figure \ref{phi-fig}.  This research
was supported in part by NSF Grants DMS-0504608 and DMS-0804472, as
well as an NSF Graduate Research Fellowship.


\bibliographystyle{plain}
\bibliography{allpapers}

\end{document}